\documentclass[12pt]{amsart}

\usepackage{amsfonts, amssymb, amscd}
\usepackage{verbatim}
\usepackage{eucal}
\usepackage{amssymb}
\usepackage{mathrsfs}
\usepackage{graphicx}
\usepackage{psfrag}

\def\Vol{{\rm Vol}}

\def\HH{{\rm H}}

\def\OO{{\mathcal O}}
\def\LL{{\mathcal L}}

\def\ZZ{{\mathbb Z}}
\def\RR{{\mathbb R}}

\def\PP{{\mathbb P}}
\def\CC{{\mathbb C}}
\def\QQ{{\mathbb Q}}

\addtocounter{footnote}{1}

\newtheorem{lemma}{Lemma}[section]
\newtheorem{theorem}[lemma]{Theorem}
\newtheorem{corollary}[lemma]{Corollary}
\newtheorem{conjecture}[lemma]{Conjecture}
\newtheorem{proposition}[lemma]{Proposition}
\theoremstyle{definition}
\newtheorem{definition}[lemma]{Definition}

\newtheorem{question}[lemma]{Question}
\newtheorem{remark}[lemma]{Remark}

\theoremstyle{remark}
\newtheorem*{proof*}{Proof}
\numberwithin{equation}{section}

\title[On the conjecture of King]{On the conjecture
of King for smooth toric Deligne-Mumford stacks}

\author{Lev Borisov and Zheng Hua}
\address{Department of Mathematics \\ University of Wisconsin \\
  Madison \\ WI \\ 53706 \\ USA\\{\tt borisov@math.wisc.edu}\\
{\tt hua@math.wisc.edu}  }

\thanks{Lev Borisov has been partially supported
by the National Science Foundation under grant No.\ DMS-0456801.}

\begin{document}

\begin{abstract}
We construct full strong exceptional
collections of line bundles on smooth toric Fano Deligne-Mumford
stacks of Picard number at most two and of 
any Picard number in dimension two. 
It is hoped that the approach of this paper 
will eventually
lead to the proof of the existence of such collections
on all smooth toric nef-Fano Deligne-Mumford stacks.
\end{abstract}

\maketitle

\section{Introduction}
It has been suggested by Alastair King in \cite{King} that
every smooth toric variety has a full strong exceptional collection 
of line bundles. While this turned  out to be false, see \cite{Perling1},
it is still natural to conjecture that every smooth
nef-Fano toric variety
possesses such a collection, and there is some numerical
evidence towards it. Here a variety is called nef-Fano (also often
referred to as weak Fano) if its anticanonical class is nef and big,
though not necessarily ample.
We refer the readers to the introduction
section of
\cite{CrawSmith} for the more detailed exposition of this area.
In this  paper we 
propose to extend the conjecture of King to smooth toric 
Deligne-Mumfod stacks, which were defined in \cite{BCS}.

\smallskip\noindent
{\bf Conjecture \ref{fsec.conj}.}
\emph{Every smooth nef-Fano toric DM stack possesses a full
strong exceptional collection of line bundles.}
\smallskip

There are multiple advantages to working with stacks rather 
than varieties in the context of this conjecture. 
Smooth toric DM stacks behave like smooth toric varieities
in many ways, so it is plausible that if Conjecture
\ref{fsec.conj} holds in the case of varieties, then it holds
in this more general setting, at least when the stacks 
are generically schemes. On the other hand, 
while there are only finitely many smooth toric nef-Fano varieties 
in any given dimension, there are infinitely many smooth
toric Fano stacks, and they correspond to nice
combinatorial data of simplicial convex lattice polytopes.
Consequently, working with stacks allows one to test the
conjecture on numerous families of examples, and to
concentrate on the more essential features
of the Fano condition. Last, but not least, stacks appear
naturally from the point of view of homological mirror
symmetry. For example, it is natural to 
try to extend the work of \cite{Abouzaid} to this generality,

We have been able to construct full strong exceptional
collections of line bundles for all smooth toric Fano DM 
stacks $\PP_{\bf\Sigma}$ of Picard number
at most two or, significantly, of dimension at most two. 
This dimension two case is of special importance
since it is related to (noncompact) toric Calabi-Yau
stacks of dimension three.

The main ingredient of the argument is 
a convex polytope $P$ in ${\rm Pic}(\PP_{\bf\Sigma})\otimes \RR$
which is to be thought of as a window into
${\rm Pic}(\PP_{\bf\Sigma})$.
For a generic point $p\in {\rm Pic}(\PP_{\bf\Sigma})\otimes \RR$,
we define the strong exceptional collection $S$ as the 
set of line bundles such that the corresponding points in
${\rm Pic}(\PP_{\bf\Sigma})\otimes \RR$ lie in $p+P$.
In other words, $S$ is the set of line bundles that we can 
see through the $P$ window, when it is shifted by $p$.
We then move $p$ and $p+P$, and as new line bundles 
appear in the window, we use Koszul complexes 
to generate them from the line bundles that we have already
seen. In the Picard number one case, $P$ is a segment, and in
the Picard number two case it is a parallelogram,
irrespective of the dimension of $\PP_{\bf\Sigma}$.
In the case of Picard number three and dimension two,
the polytope $P$ is a $10$-gonal prism, and careful
arguments of convex geometry are needed to establish
its various properties. For arbitrary Picard numbers and
dimension two, $P$ is a zonotope, i.e. a Minkowski sum
of line segments.

One key property of the polytope $P$ is that the differences
between all pairs of its interior points give acyclic
line bunldles. To prove this, we introduce the notions of strong
acyclicity and forbidden cones, see Definition \ref{forbidden}.
This approach follows the calculations of Danilov \cite{Danilov}
and is similar to the recent work of Perling \cite{Perling2} in 
the scheme setting. The notion of strong acyclicity allows one 
to reduce the calculations to questions of convex geometry.

The paper is organized as follows. In Section \ref{sec2}
we briefly review the definition of smooth toric Deligne-Mumford
stacks. In Section \ref{sec3} we describe line bundles
on these stacks and state the main Conjecture \ref{fsec.conj}.
In Section \ref{sec4} we give a combinatorial description
of cohomology of a line bundle on a smooth toric Deligne-Mumford
stack and introduce the notions of strong acyclicity and forbidden
cones. In Section \ref{sec5} we describe the construction
in the cases of Picard number one and two. Sections
\ref{sec.dP.prelim} and \ref{sec.dP} treat the case of smooth 
toric del Pezzo DM stacks. The former
section contains the calculations in the quotient of 
${\rm Pic}_\RR(\PP_{\bf\Sigma})$ by the span 
of the canonical class $\RR K$,
and the latter finishes the argument. Section \ref{sec.ex}
describes our construction in the case of dimension two 
and Picard number three.  Finally, in Section \ref{comments}
we briefly describe the difficulties one encounters when 
one tries to extend the method to higher dimension.

{\bf Acknowledgments.} We learned about this conjecture
(in the case of varieties) from Alastair Craw. We would like 
to thank Angelo Vistoli for helpful remarks and Rosa Mir\'o-Roig 
for a useful reference. 

\section{An overview of toric DM stacks}\label{sec2}
Let $N$ be a finitely generated abelian group and let 
$\Sigma$ be a complete
simplicial fan in $N$ (which is simply a pullback of a simplicial fan $\Sigma_{\rm free}$ in $N_{\rm free}=
N/{\rm torsion}(N)$). If 
one chooses
a non-torsion element $v$ in 
each of the one-dimensional cones of $\Sigma$, one gets 
the data of a
complete stacky fan ${\bf\Sigma}=(\Sigma,\{v_i\})$, 
see \cite{BCS}. To these data one 
can then associate a smooth toric Deligne-Mumford stack
$\PP_{\bf\Sigma}$ 
whose coarse moduli space is the proper simplicial toric
variety given by $\Sigma_{\rm free}$.

We will assume from now on that $N$ has no torsion,
to simplify the discussion, although it appears that 
the general case is not very different. This 
assumption will allow us to avoid 
the technicalities of the derived Gale duality of \cite{BCS}.
The toric Deligne-Mumford stack $\PP_{\bf\Sigma}$ is 
obtained by a stacky version of the Cox's homogeneous
coordinate ring construction of 
\cite{Cox}. More specifically, if $\Sigma$
has $n$ one-dimensional cones, then we have a map
$$
\rho\colon \ZZ^n\to N
$$
defined by $(l_1,\ldots, l_n)\mapsto \sum_i l_i v_i$ where
$v_i$ are the chosen elements of ${\bf\Sigma}$. 
We dualize to get an injection 
$$
\rho^*\colon N^*\to \ZZ^n
$$
and we denote the cokernel of $\rho^*$ by ${\rm Gale}(N)$.
The group ${\rm Gale}(N)$  is a finitely generated
abelian group of rank $n-{\rm rk}(N)$ and it
has torsion if and only if $\rho$ is not surjective. We 
define the abelian complex algebraic group $G$ by 
$$
G:={\rm Hom}({\rm Gale}(N),\CC^*).
$$
The group $G$ is (non-canonically) isomorphic to a product of 
$(\CC^*)^{n-{\rm rk}(N)}$ and a finite abelian group.

The map $\rho^*$ induces an injection 
$$G\subseteq (\CC^*)^n$$
and an element $(\lambda_1,\ldots,\lambda_n)\in (\CC^*)^n$
lies in $G$ if and only if 
$$
\prod_{i=1}^n \lambda_i^{w\cdot v_i}=1
$$
for all $w\in N^*$, where $\cdot$ denotes the natural
pairing.

Consider the open set $U$ in $\CC^n$ defined as follows.
A point $(z_1,\ldots,z_n)$ lies in $U$ if and only if 
there exists a cone in $\Sigma$ which contains 
all $v_i$ for which $z_i=0$.
It turns out that the action of $G$ has only finite isotropy 
subgroups on $U$, and  
$\PP_{\bf\Sigma}$ is then defined as the stack quotient
$[U/G]$, see \cite{BCS}. 


\section{Derived category of toric stacks and King's conjecture}
\label{sec3}
We keep the notations from the previous section.
In this section we will describe some of the known results
about the derived category of coherent sheaves  on $\PP_{\bf\Sigma}$ and will formulate the  conjecture,
whose original version is due to Alastair King, \cite{King}.
See \cite{CrawSmith} for a short review of the related results.

The category of coherent sheaves on $\PP_{\bf\Sigma}$ 
is equivalent to the category of 
$G$-equivariant sheaves on $U$, see \cite[7.12]{Vistoli}.
In particular, the line bundles on $\PP_{\bf\Sigma}$ have the 
following explicit description.

\begin{definition}\label{linebundles}
For each $(r_1,\ldots,r_n)\in\ZZ^n$ consider the trivial line bundle 
$\CC\times U \to U$ with the $G$-linearization 
$G\times \CC\times U\to \CC\times U$  given by 
$$
((\lambda_1,\ldots,\lambda_n),t, (z_1,\ldots,z_n))
\mapsto (t\prod_{i=1}^n \lambda_i^{r_i},
(\lambda_1z_1,\ldots,\lambda_nz_n)).
$$
By \cite{Vistoli}, this gives a line bundle on $\PP_{\bf\Sigma}$.
We will denote it by $\OO(\sum_i r_iE_i)$.
\end{definition}

\begin{remark}
We will implicitly identify line bundles and invertible 
sheaves of their regular sections throughout the paper.
\end{remark}

\begin{proposition}\label{allaregiven}
All line bundles on $\PP_{\bf\Sigma}$ are 
given by the construction of Definition \ref{linebundles}.
The Picard group of $\PP_{\bf\Sigma}$ is 
isomorphic to the quotient of $\ZZ^n$ with basis $(E_i)$
by the subgroup of elements of the form 
$$
\sum_{i=1}^n (w\cdot v_i)E_i
$$
for all $w\in N^*$.
\end{proposition}

\begin{proof}
The line bundles on $\PP_{\bf\Sigma}$ correspond
to $G$-equivariant line bundles on $U$. 
The open set $U$ is a smooth toric variety, so its Picard
group is generated by invariant divisors $z_i=0$, which 
are clearly trivial. Consequently,
every line bundle on $U$ can be trivialized. 
To classify line bundles on $\PP_{\bf\Sigma}$ one
thus needs to classify the $G$-linearizations of the 
trivial line bundle $\CC\times U\to U$. 

For every $g\in G$, we have 
$$g\colon (t,{\bf z})\mapsto (t\, r_g({\bf z}),g{\bf z}).$$
The function $r_g$ is an invertible regular function on $U$.
Since $U$ is obtained from $\CC^n$
 by removing subspaces of codimension
at least two, the ring of regular functions on $U$ is 
$\CC[z_1,\ldots,z_n]$, and any invertible regular function
on $U$ is 
a nonzero constant. Then the definition of $G$-linearization shows
that the map $G\to \CC^*$ given by $g\mapsto r_g$ 
gives a line bundle if and only if it is a character of $G$.
The characters of $G$ are given by ${\rm Gale}(N)$, which
has the desired description in terms of $E_i$.
\end{proof}

The following result has been shown in \cite{BH}.
\begin{theorem}\label{bh}
The derived category of $\PP_{\bf\Sigma}$ is generated
by line bundles. 
\end{theorem}

\begin{proof}
See Corollary 4.8 of \cite{BH}.
\end{proof}

The focus of this paper is on constructing, in some
special cases, collections of line bundles on
$\PP_{\bf\Sigma}$ which satisfy certain cohomological
properties.
\begin{definition}\label{secdef}
A sequence of line bundles $(\LL_1,\ldots,\LL_r)$ 
on  $\PP_{\bf\Sigma}$ is called a strong exceptional
collection if
$$
{\rm Ext}^i_{\PP_{\bf\Sigma}}(\LL_{j_1},\LL_{j_2})=0
$$
unless $i=0$ and $j_1\leq j_2$.
\end{definition}

\begin{remark}\label{anyorder}
A subset $S$ of 
${\rm Pic}(\PP_{\bf\Sigma})$ can be indexed to form a
strong exceptional
collection, as long as 
${\rm Ext}^i_{\PP_{\bf\Sigma}}(\LL_1,\LL_2)=0$
for all $i>0$ and all $\LL_1$ and $\LL_2$ in $S$. Indeed,
the existence of nonzero 
${\rm Hom}_{\PP_{\bf\Sigma}}(\LL_1,\LL_2)$ induces
a partial order on the set $S$, which can 
then be extended to a linear order. 
\end{remark}

\begin{definition}\label{fsec}
A finite set $S$ of line bundles 
on $\PP_\Sigma$ is called a full strong exceptional
collection if 
$${\rm Ext}^i_{\PP_{\bf\Sigma}}(\LL_1,\LL_2)=0
~~{\rm for~all~}i>0{\rm~and~all~}\LL_1,\LL_2\in S,$$
and the line bundles in $S$ generate the derived category
of $\PP_{\bf\Sigma}$.
\end{definition}

It is only natural to ask the following question.
\begin{question}\label{big}
Does $\PP_{\bf\Sigma}$ possess a full strong exceptional
collection of line bundles?
\end{question}

\begin{remark}
Kawamata has shown that the derived category of 
$\PP_{\bf\Sigma}$ possesses a full exceptional collection
of objects, see \cite{Kawamata}. 
In his construction, the objects are typically
sheaves rather than line bundles, and the collection
is only exceptional, rather than strong exceptional
(some nontrivial higher $\rm Ext$ spaces are allowed).
\end{remark}

\begin{remark}
There is an example of a smooth toric surface which 
does not admit a full strong exceptional collection of line bundles,
see \cite{Perling1}. A quick review of the related results
can be found in \cite{CrawSmith}.
It has been subsequently suggested, 
that in the case of varieties a sufficient condition for the
positive answer to Question \ref{big} is that $\PP_{\Sigma}$
is a Fano variety. We will argue in this paper that
it is reasonable to expect that Question \ref{big} has a 
positive answer for all nef-Fano Deligne-Mumford stacks,
to be defined below. 
\end{remark}

\begin{definition}\label{nefFano}
A toric Deligne-Mumford stack $\PP_{\bf\Sigma}$
is called Fano if the chosen points $v_i$ are precisely the 
vertices of a simplicial convex polytope in $N_\RR$.
More generally, it is called nef-Fano if all $v_i$ 
lie on the boundary 
of  $$\Delta={\rm ConvexHull}(v_1,\ldots,v_n)$$ 
but are not necessarily its vertices,
nor is $\Delta$ assumed to be simplicial.
\end{definition}

\begin{remark}
The terminology of Definition  \ref{nefFano} is 
justified as follows. A positive power of the 
anticanonical line bundle
on $\PP_{\bf\Sigma}$ is a pullback of a line bundle
on the coarse moduli space. The stack $\PP_{\bf\Sigma}$
is Fano (resp. nef-Fano) if the corresponding Cartier divisor
is ample (resp. nef and big). Since we do not use this interpretation
of the definition, we leave the verification of the above statement
to the reader.
\end{remark}

\begin{remark}
In dimension two case, we call the Fano stacks del Pezzo,
in accordance with the common terminology for varieties.
\end{remark}

We are now ready to formulate the stack version of 
King's conjecture.
\begin{conjecture}\label{fsec.conj}
Every smooth nef-Fano toric DM stack possesses a full
strong exceptional collection of line bundles.
\end{conjecture}

\begin{remark}\label{what}
From the general theory of exceptional collections,
the number of elements in a strong exceptional collection of line
bundles equals the rank of $K$-theory. For a smooth 
toric nef-Fano DM stack this rank in turn
equals ${{\rm rk} (N)}! \Vol(\Delta)$, 
where the volume is normalized so that 
the volume of $N_\RR/N$ is one, see for example \cite{BH2}.
\end{remark}


\section{Strongly acyclic line bundles}\label{sec4}

The following rather standard calculation provides a description 
of cohomology of a line bundle $\LL$ on $\PP_{\bf\Sigma}$.
For every 
${\bf r}=(r_i)_{i=1}^n\in \ZZ^n$ we denote by ${\rm Supp}
({\bf r})$ the simplicial complex on
$n$ vertices $\{1,\ldots,n\}$ which consists of all subsets 
$J\subseteq \{1,\ldots,n\}$ such that
$r_i\geq 0$ for all $i\in J$ and there exists a
cone of $\Sigma$  that contains all $v_i, i\in J$.
We will abuse notation to 
also denote by ${\rm Supp}({\bf r})$ the subfan of $\Sigma$
whose cones are the minimum cones of $\Sigma$ that contain all $v_i,i\in J$ for all subsets $J$
as above. It should be clear from the context whether 
${\rm Supp}({\bf r})$ refers to the simplicial complex or to
its geometric realization as a subfan of $\Sigma$. For example,
if all coordinates $r_i$ are negative then the simplicial
complex ${\rm Supp}({\bf r})$
consists of the empty set only, and its geometric
realization is the zero cone of $\Sigma$. 
In the other extreme case, if all $r_i$
are nonnegative then the simplicial complex 
${\rm Supp}({\bf r})$ encodes the fan $\Sigma$, which is its 
geometric realization.

\begin{proposition}\label{hp}
The cohomology $\HH^p(\PP_{\bf\Sigma},\LL)$
is isomorphic to the direct sum over all 
${\bf r}=(r_i)_{i=1}^n$ 
such that  $\OO(\sum_{i=1}^n r_i E_i)\cong \LL$ 
of the $({\rm rk}(N)-p)$-th reduced homology of the simplicial
complex ${\rm Supp}({\bf r})$.
\end{proposition}

\begin{proof}
Consider the left exact functor ${\rm H}^0(\PP_{\bf\Sigma},\bullet)$ on the
category of $G$-equivariant quasi-coherent sheaves 
on $U$. It sends a $G$-equivariant sheaf on $U$ to the space 
of its $G$-invariant global sections. Hence, it is the composition
of the functor of global sections and the functor 
of taking $G$-invariants. Since $G$ is reductive,
the latter is exact, consequently, 
$$
{\rm H}^p(\PP_{\bf\Sigma},\LL)=
({\rm H}^p(U,\LL))^{G}.
$$
Recall that $\LL\cong \OO_U$ if one ignores the 
action of $G$. We can calculate $H^p(U,\OO)$ by resolving 
$\OO$ via a toric \v{C}ech complex. Specifically,
$U$ is a toric variety whose toric affine charts $U_\sigma$ 
are indexed by $\sigma\in\Sigma$. A point
$(z_1,\ldots, z_n)$ lies in $U_\sigma$ if and only if 
all $v_i$ for which $z_i=0$ lie in $\sigma$. Consequently,
$\Gamma(U_\sigma,\OO)$ 
has a monomial basis of $\prod_i z_i^{a_i}$
with $a_i\geq 0$ for all $v_i\in\sigma$ and $a_i\in \ZZ$ 
otherwise. The cohomology of $\OO$ on $U$ is 
naturally isomorphic to the cohomology of the 
toric \v{C}ech complex
\begin{equation}\label{toricech}
0\to \bigoplus_{\stackrel{\sigma\in\Sigma,}{{\rm dim}\sigma={\rm rk}(N)}}
\Gamma(U_\sigma,\OO) 
\to \bigoplus_{\stackrel{\sigma\in\Sigma,}{{\rm dim}\sigma=
{\rm rk}(N)-1}}
\Gamma(U_\sigma,\OO) 
\to\cdots\to
\Gamma(U_{\{0\}},\OO)\to 0.
\end{equation}
The maps in this complex are direct sums of the maps from $\Gamma(U_\sigma,{\mathcal O})$
to $\Gamma(U_{\sigma'},{\mathcal O})$ which 
are zero unless
$\sigma'$ is a codimension one face of $\sigma$. In
this case the map is, up to a sign, the restriction
map with the sign determined as follows. If
$$
\RR_{\geq 0} \sigma = \bigoplus_{j=1}^{\dim \sigma} 
\RR_{\geq 0} v_{i_j}, ~~
\RR_{\geq 0} \sigma' = \bigoplus_{j=1,j\neq k}^{\dim \sigma} 
\RR_{\geq 0} v_{i_j}~
$$
with $i_1<\ldots <i_{\dim \sigma}$, then the sign is $(-1)^k$.

This complex is graded by the characters of $(\CC^*)^n$,
i.e. by multidegree of the monomials. For any given 
collection ${\bf r}=(r_i)_{i=1}^n\in\ZZ^n$, the graded piece
of the complex
\eqref{toricech} at multidegree $\bf r$ is precisely the 
reduced homology complex of ${\rm Supp}({\bf r})$.
Indeed, the space of sections of $\mathcal O$ on $U_\sigma$
contains a one-dimensional graded piece $\CC\prod_i z_i^{r_i}$ if and only if $\sigma$ contains no
$v_i$ for which $r_i < 0$, i.e. the set $J$ of $i$ such that $v_i\in\sigma$ is a subset of the simplicial complex
${\rm Supp}({\bf r})$. Moreover, the maps in
\eqref{toricech} are the same as
in the reduced homology complex of ${\rm Supp}({\bf r})$.

It remains to show that taking $G$-invariants amounts
to only picking ${\bf r}$ with $\OO(\sum_{i=1}^nr_iE_i)\cong \LL$,
which follows from Definition \ref{linebundles} and the description
of $G$ in Section \ref{sec2}.
\end{proof}

\begin{remark}
For example, $\HH^0(\LL)$ only comes from $\bf r$ for 
which ${\rm Supp}({\bf r})$ is the entire fan $\Sigma$,
i.e. $\OO(\sum_{i=1}^n r_i E_i)\cong \LL$ with
${\bf r}\in \ZZ_{\geq 0}^n$.
In the other extreme case $\HH^{{\rm rk}(N)}(\LL)$ only 
appears when the simplicial complex 
${\rm Supp}({\bf r})=\{\emptyset\}$, i.e. when
$\OO(\sum_{i=1}^n r_i E_i)\cong \LL$ with all 
$r_i\leq -1$.
\end{remark}


As usual, we will call a line bundle acyclic if all of
its higher cohomology
groups vanish. Based on Proposition \ref{hp} we can describe all 
acyclic line bundles on $\PP_{\bf\Sigma}$ as follows.
For every subset $I\subseteq \{1,\ldots,n\}$ consider
the simplicial complex $C_I$ which encodes the cones of $\Sigma$, such that the indices
of all rays of the cone lie in $I$. In other words, this complex
is ${\rm Supp}({\bf r})$ where $r_i=-1$ for $i\not\in I$ and 
$r_i=0$ for $i\in I$. 
\begin{proposition}\label{iff}
Consider all proper subsets $I\subset \{1,\ldots,n\}$ such 
that  $C_I$ has nontrivial reduced homology. For each such
subset consider the set of line bundles on $\PP_{\bf\Sigma}$
of the form
$$
\OO(-\sum_{i\not\in I}E_i + \sum_{i\in I}
r_iE_i
-\sum_{i\not\in I}
r_iE_i)$$
where $r_i\in \ZZ_{\geq 0}$ for all $i$.
Then a bundle $\LL$ is acyclic if and only if it does not 
lie in any of the above sets.
\end{proposition}

\begin{proof}
This is an immediate corollary of Proposition \ref{hp}.
\end{proof}

It is in principle rather difficult to apply the above criterion.
We can provide a more manageable 
sufficient condition of acyclicity as follows. 
Consider
${\rm Pic}_\RR(\PP_{\bf\Sigma})\colon= 
{\rm Pic}_{\ZZ}(\PP_{\bf\Sigma})\otimes \RR$. 
We can think of it as a quotient of $\RR^n$ with
basis elements given by $E_i$.
\begin{definition}\label{forbidden}
For each proper subset $I\subset \{1,\ldots,n\}$ 
such that $C_I$ has nontrivial reduced homology define
the \emph{forbidden point}
$$q_I=-\sum_{i\not\in I}E_i \in{\rm Pic}_\RR(\PP_{\bf\Sigma})$$
Define the \emph{forbidden cone} $F_I\subseteq {\rm Pic}_\RR(\PP_{\bf\Sigma})$
by 
$$
F_I=q_I+\sum_{i\in I} \RR_{\geq 0} E_i - \sum_{i\not\in I}
\RR_{\geq 0} E_i.
$$ 
A line bundle $\LL$ is called \emph{strongly acyclic}
if its image in ${\rm Pic}_\RR(\PP_{\bf\Sigma})$ does not
lie in any of the forbidden cones.
\end{definition}

\begin{proposition}\label{sacisac}
Every strongly acyclic line bundle is acyclic.
\end{proposition}

\begin{proof}
This statement follows immediately from Proposition \ref{iff}.
\end{proof}

\begin{remark}
The concept of strong acyclicity has several advantages 
over  the usual acyclicity. For example, it can be checked for
by looking at a finite set of inequalities. It would be 
interesting to figure out the geometric meaning of
strong acyclicity and to see if this notion can be defined beyond
the toric case.
\end{remark}

\begin{remark}
An example of a line bundle which is acyclic but not
strongly acyclic is ${\mathcal O}(-6)$ on
the weighted projective line with weights $2$ and $3$.
Here the Picard group is isomorphic to $\ZZ$ with images
of ${\mathcal O}(E_1)$ and 
${\mathcal O}(E_2)$ given by ${\mathcal O}(2)$ and 
${\mathcal O}(3)$ respectively. It is impossible to
write ${\mathcal O}(-6)={\mathcal O}(r_1E_1+r_2E_2)$ 
with negative integer $r_i$, which means that 
${\mathcal O}(-6)$ is acyclic. On the other 
hand the forbidden cone $F_{\emptyset}$ contains
the images of all ${\mathcal O}(k)$ with $k\leq -5$, 
so ${\mathcal O}(-6)$ is not strongly acyclic.
\end{remark}

\section{The case of ${\rm rk(Pic)}\leq 2$}\label{sec5}
In this section we will argue that Conjecture \ref{fsec.conj}
is true for toric Fano Deligne-Mumford stacks $\PP_{\bf\Sigma}$
with ${\rm rk}({\rm Pic}(\PP_{\bf\Sigma}))\leq 2$.

We first consider the case of 
${\rm rk}({\rm Pic}(\PP_{\bf\Sigma}))=1$. In this case 
$\Delta$ is a simplex in the lattice $N$ of rank $(n-1)$.
The only forbidden cone occurs for $I=\emptyset$,
with the corresponding forbidden point $-\sum_{i=1}^nE_i$.
Denote by 
$$\deg\colon {\rm Pic}(\PP_{\bf\Sigma})\to \ZZ$$
the linear function that takes value $1$ on the positive
generator of ${\rm Pic}(\PP_{\bf\Sigma})$. Then 
the forbidden cone is given by $x\in {\rm Pic}_\RR(\PP_{\bf\Sigma})$
such that 
$$
\deg(x)\leq -\sum_{i=1}^n \deg(E_i)= \deg(K)
$$
where $K$ is the canonical divisor.

\begin{proposition}
Consider the set $S$ of line bundles $\LL$ with
$\deg\LL \in [\deg(K)+1,0]$. Then the set $S$ 
forms a full strong exceptional collection on $\PP_{\bf\Sigma}$.
\end{proposition}

\begin{proof}
It is clear that for any two $\LL_1$ and $\LL_2$ in $S$,
the line bundle $\LL_2\otimes \LL_1^{-1}$ has degree
bigger than $\deg(K)$ and is therefore acyclic by
Proposition \ref{sacisac}.

Consider the subcategory $D$ of the derived category 
of $\PP_{\bf\Sigma}$ which is generated by $\LL$ in $S$.
In view of Theorem \ref{bh}, it suffices to show that all
line bundles on $\PP_{\bf\Sigma}$ lie in $D$.

Let us first prove this for all line bundles of nonnegative degree
by induction on $\deg(\LL)$. The base of induction $\deg(\LL)=0$
is clear. Suppose now that we have shown this for all
$\LL$ of $0\leq \deg(\LL)\leq k$. Then if 
$\LL=\OO(E)$ has degree $(k+1)$, 
consider the Koszul complex
$$
0\to \OO(E-\sum_{i=1}^n E_i) \to \ldots \to \oplus_{i=1}^n \OO(E-E_i) \to \LL\to 0.
$$
This comes from a Koszul complex on $\CC^n$ which resolves
the point $(0,\ldots,0)\not\in U$. As a result, it leads to an exact
complex on $\PP_\Sigma$, see \cite{BH}.
All but the last term of the complex are in $D$, hence
so
is $\LL$, which proves the induction step.

A similar, decreasing, induction on the degree allows 
us to handle the case of $\deg(\LL)\leq \deg(K)$, which finishes
the proof.
\end{proof}

\begin{remark}\label{numpic1}
The number of elements of $S$ equals $(-\deg(K))d$ where 
$d$ is the order of the torsion subgroup of 
${\rm Pic}(\PP_{\bf\Sigma})$. This coincides with the 
rank of the Grothendieck group of $\PP_{\bf\Sigma}$, which
is not a coincidence, but rather is expected by Remark \ref{what}.
\end{remark}

\begin{remark}
The case of Picard number one has already been settled
in \cite{Kawamata}, but we have treated it here nonetheless,
to give a unified picture of our approach.
\end{remark}

We will now move to the more challenging case of 
${\rm rk}({\rm Pic}(\PP_{\bf\Sigma}))=2$. We have 
$n$ elements $v_i$ of the lattice $N$ of rank $n-2$,
which form the set of vertices of a simplicial polytope
$\Delta$.

\begin{proposition}\label{alpha}
There exists a unique up to scaling collection of
rational numbers $\alpha_i$, such that 
$\sum_{i=1}^n \alpha_i =0$ and $\sum_{i=1}^n \alpha_i v_i=0$.
Moreover, all of the $\alpha_i$ in this relation are nonzero.
\end{proposition}

\begin{proof}
Since $\Sigma$ is a complete fan,
the vertices $v_i$ generate $N\otimes \QQ$, so
the space of linear relations on $v_i$ is of dimension two.
Since $0$ is in the convex hull of $v_i$, it can be
written as a sum of $v_i$ with nonnegative coefficients.
Hence, there is a relation on $v_i$ with $\sum_{i=1}^n\alpha_i>0$.
Consequently, the condition 
$\sum_{i=1}^n \alpha_i =0$ cuts out a dimension one subspace
of relations.

Suppose some $\alpha_i$ is zero. It means that $v_j,j\neq i$
lie in a proper affine subspace of $N\otimes \QQ$.
It then gives a supporting hyperplane of $\Delta$ which 
has $(n-1)$ points in it, in contradiction with simpliciality
of $\Delta$.
\end{proof}

We will pick one such relation
$\sum_{i=1}^n\alpha_iv_i=0$. We will denote by $I_+$ 
(resp. $I_-$) the sets of $i$ with positive $\alpha_i$ (resp. 
negative $\alpha_i$).
\begin{proposition}\label{pic2fan}
The facets of $\Delta$ are precisely convex hulls of $(n-2)$ of 
the $v_i$-s, such that one of the remaining two indices 
lies in $I_+$ and the other lies in $I_-$. 
\end{proposition}

\begin{proof}
Consider a subset $I\subset\{1,\ldots,n\}$ of 
cardinality $(n-2)$. The convex hull of $v_i, i\in I$  
does \emph{not} form a face of $\Delta$ if and only if the segment through remaining two vertices intersects the affine span of 
this set. This is equivalent to the existence of  a relation
$$
 \sum_{i\in I}\beta_i v_i= \sum_{j\not\in I}\beta_j v_j
$$
with $\sum_{i\in I}\beta_i =1=\sum_{j\not\in I}\beta_j$
and with the two $\beta_j$ both positive. By comparing 
with the result of Proposition \ref{alpha}, this implies
that the complement of $I$ is a subset of $I_+$ or of $I_-$.
Conversely, for any two indices $j_1,j_2$ in $I_+$ or $I_-$, 
one can move $\alpha_{j_1}v_{j_1}+\alpha_{j_2}v_{j_2}$ to
the right hand side in the equation $\sum_{i=1}^n \alpha_iv_i=0$
and then divide by $-(\alpha_{j_1}+\alpha_{j_2})$ to get
$$\sum_{i\neq j_1,j_2}\beta_iv_i =
\beta_{j_1}v_{j_1}+\beta_{j_2}v_{j_2}$$ 
with $\sum_{i\neq j_1,j_2}\beta_i = 1=\beta_{j_1}+\beta_{j_2}$
and positive $\beta_{j_1}$ and $\beta_{j_2}$.
\end{proof}

\begin{corollary}\label{cor}
A subset $I$ of $\{1,\ldots,n\}$ corresponds to a face of 
$\Delta$ if and only if the complement of $I$ is not 
contained in $I_+$ or $I_-$. In addition, the sets
of $I_+$ and $I_-$ have at least two elements each.
\end{corollary}

\begin{proof}
The first statement follows immediately from Proposition
\ref{pic2fan}. If $I_+$ or $I_-$ has only one element,
then the corresponding $v_i$ 
does not lie in any face of $\Delta$.
\end{proof}

The following proposition classifies the forbidden cones
in this case. 
\begin{proposition}\label{forbpic2}
There are precisely three forbidden cones,
which correspond to the subsets 
$\emptyset$, $I_+$ and $I_-$ of $\{1,\ldots,n\}$.
\end{proposition}

\begin{proof}
Suppose that both $I$ and its complement $\bar I$ 
intersect $I_+$ nontrivially. Pick $i\in I\cap I_+$.
By Corollary \ref{cor} the simplicial complex $C_I$
is a cone over $i$ (i.e. $i$ can be added to any of its subsets) 
and is thus acyclic. Similarly,
if $I$ and $\bar I$ intersect $I_-$ nontrivially,
then $C_I$ is acyclic. 

It remains to observe that for $I$ that are equal to
$I_\pm$ the corresponding simplicial complex $C_I$
has a geometric realization of the sphere and 
consequently has nontrivial reduced homology.
\end{proof}

For what follows we pick and fix a collection of positive numbers $r_i, i=1,\ldots, n$, 
such that $\sum_i r_i=1$ and $\sum_i r_iv_i=0$. This collection
gives a linear function $f$ on
${\rm Pic}_\RR(\PP_{\bf\Sigma})$ with $f(E_i)=r_i$.
Similarly, we define a linear function $\alpha$ 
with $\alpha(E_i)=\alpha_i$
from Proposition \ref{alpha}.
Consider the parallelogram $P$ in 
${\rm Pic}_\RR(\PP_{\bf\Sigma})$ which is given
by the inequalities 
$$
\vert f(x) \vert \leq \frac 12,\hskip .25in
\vert \alpha(x)\vert \leq \frac 12\sum_{i\in I_+} \alpha_i.
$$

\begin{proposition}\label{interior}
The interior of the parallelogram $2P$ contains 
no points from the forbidden cones. The only 
points on the boundary of $2P$ that lie in
the forbidden cones are 
$-\sum_{i=1}^n E_i$, $-\sum_{i\in I_-} E_i$ 
and $-\sum_{i\in I_+}E_i$, see Figure \ref{FigPic2-1}.
\end{proposition}

\begin{figure}[tbh] 
\begin{center}
\includegraphics[scale = .7]{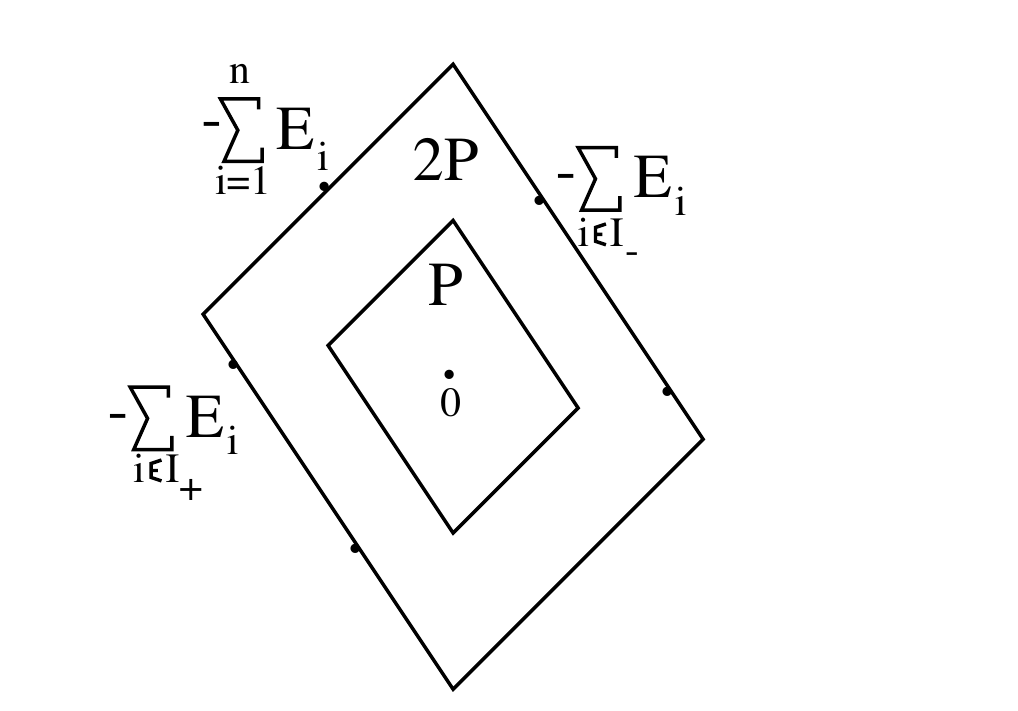}
\end{center}
\caption{\label{FigPic2-1}}
\end{figure}

\begin{proof}
There are three forbidden cones, described in
Proposition \ref{forbpic2}. We will show that for each
of these cones the corresponding forbidden point lies on the 
side of $2P$ with the side giving a supporting hyperplane 
of the cone. For $x$ in the cone
$$
-\sum_{i\in I_-} E_i +\sum_{i\in I_+} \RR_{\geq 0} E_i
-\sum_{j\in I_-} \RR_{\geq 0} E_j, 
$$
we have 
$$\alpha(x) = -\sum_{i\in I_-} \alpha_i 
+ \sum_{i\in I_+} t_i\alpha_i - \sum_{j\in I_-}t_j\alpha_j
\geq 
 -\sum_{i\in I_-} \alpha_i 
 =\sum_{i\in I_+}\alpha_i,
$$
with the equality if and only if all $t_i$ and $t_j$ are zero.
The other two cones are handled similarly.
\end{proof}

\begin{proposition}
Consider the four points 
$$
\pm\frac 12 \sum_{i\in I_+} E_i,
\pm\frac 12 \sum_{i\in I_-} E_i.
$$
They lie on two opposite sides of $P$.
Moreover, each of the opposite sides of $P$ can be 
subdivided into a pair of segments with these points 
as centers, as in Figure \ref{FigPic2-2}.
\end{proposition}

\begin{figure}[tbh] 
\begin{center}
\includegraphics[scale = .7]{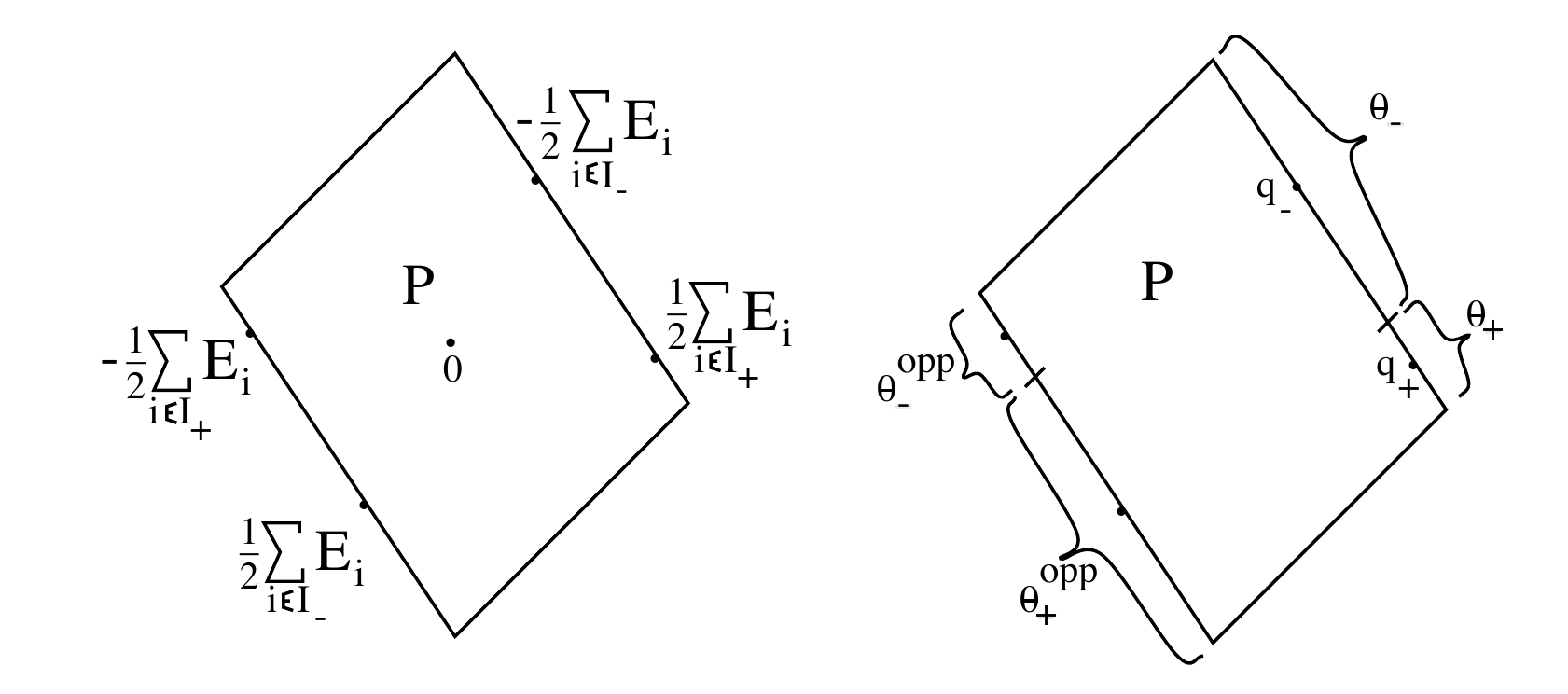}
\end{center}
\caption{\label{FigPic2-2}}
\end{figure}

\begin{proof}
In view of central symmetry of $P$ it suffices 
to show that $q_+=\frac 12 \sum_{i\in I_+} E_i$ and
$q_-=-\frac 12 \sum_{i\in I_-} E_i$ lie on its sides. 
It is clear that
$$
\alpha(q_\pm)=\frac 12\sum_{i\in I_+} \alpha_i,
$$
so it remains to check $f(q_\pm)$.
We have 
$$-\frac 12 =-\frac 12 \sum_{i=1}^nr_i <
-\frac12\sum_{i\in I_-} r_i = f(q_-)<0
<\frac 12\sum_{i\in I_+} r_i = f(q_+)<\frac12.
$$
To show the last statement, observe that 
$f(q_+) - f(-q_-) = -\frac 12$,
so the distance between the two points on the side
of $P$ is half the length of the side of $P$. 
\end{proof}

We will denote the four segments on the sides of $P$
by $\theta_{\pm}$ and $\theta_{\pm}^{\rm opp}$, 
see Figure \ref{FigPic2-2}.
The following proposition is crucial.
\begin{proposition}\label{pic2move}
Let $q$ be a point in the interior of the segment 
$\theta_{\pm}$. Then $q\mp\sum_{i\in I_\pm}E_i$
lies in the interior of the segment 
$\theta_{\mp}^{\rm opp}$, and for any proper nonempty
subset $J\subset I_\pm$ the point $q\mp\sum_{i\in J}E_i$
lies in the interior of $P$.
\end{proposition}

\begin{proof}
Since $2q_{\pm}=\pm\sum_{i\in I_\pm}E_i$,
and $\theta_{\pm}$ has the same length as 
$\theta_{\mp}^{\rm opp}$,
the translate of the interior of $\theta_{\pm}$ by
$\mp\sum_{i\in I_\pm}E_i$ is the interior 
of $\theta_{\mp}^{\rm opp}$.
For each $J\subset I_\pm$
the values of $f(q\mp\sum_{i\in J}E_i)$ and 
$\alpha(q\mp\sum_{i\in J}E_i)$ are in between those for $J=\emptyset$ and $J=I_\pm$, in view of the signs of $r_i$
and $\alpha_i$. This shows that 
$q\mp\sum_{i\in J}E_i$ is in the interior of $P$.
\end{proof}

We are now ready to construct a strong exceptional collection
$S$ in ${\rm Pic}(\PP_{\bf\Sigma})$. Pick a generic point 
$p\in{\rm Pic}_\RR(\PP_{\bf\Sigma})$
so that the lines along the sides of the parallelogram $p+P$ 
do not 
contain any points from ${\rm Pic}_\QQ(\PP_{\bf\Sigma})$.
\begin{theorem}\label{Pic2}
The set $S$ of line bundles $\LL$ such that their
image in ${\rm Pic}_\RR(\PP_{\bf\Sigma})$ lies in
$P+p$ forms a full strong exceptional collection on
$\PP_{\bf\Sigma}$.
\end{theorem}

\begin{proof}
First of all, we will show that this set forms a strong 
exceptional collection. For this it suffices to show 
that the difference of any two points in the interior of $p+P$
lies outside of the forbidden cones. Since $p+P-(p+P)=2P$,
this statement follows from Proposition \ref{interior}. 

In view of Theorem \ref{bh}, we now need to show that the
category $D$ generated by the line bundles from $S$
contains all line bundles.  At the first step of the construction
we will move the polytope $p+P$ by moving
the 
point $p$ in the line with constant $f(p)$. We claim 
that the 
newly appearing line bundles lie in $D$. Let us first 
show it for the direction indicated by the arrow in
Figure \ref{FigPic2-3}.

\begin{figure}[tbh] 
\begin{center}
\includegraphics[scale = .7]{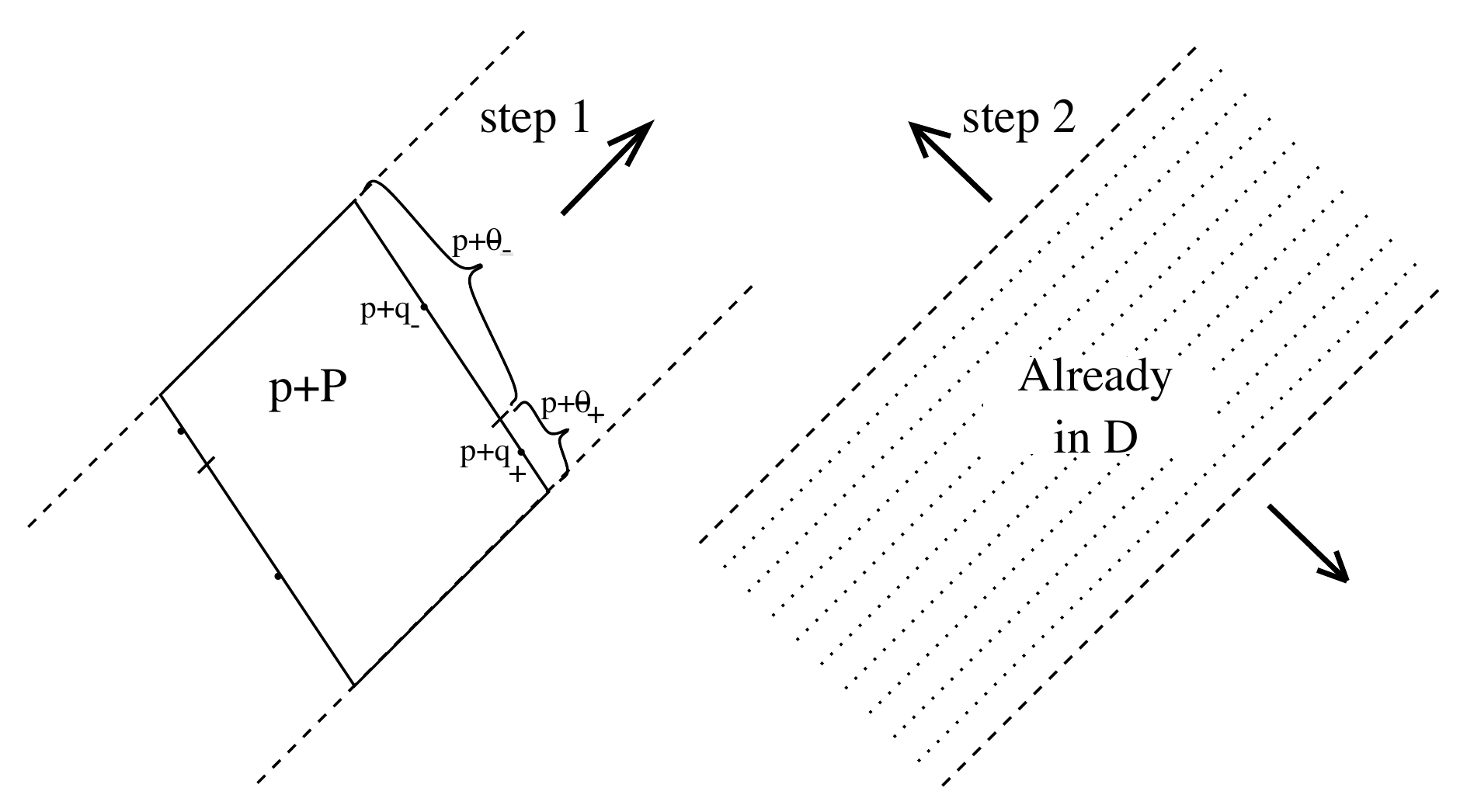}
\end{center}
\caption{\label{FigPic2-3}}
\end{figure}

Every time that the image in ${\rm Pic}_\RR(\PP_{\bf\Sigma})$
of a line bundle $\LL=\OO(E)$ fits into $p+P$, this image
will be in the interior of $p+\theta_\pm$, since we can assure
that the intersection point of $p+\theta_+$ and $p+\theta_-$ 
is moving
along a non-rational line by a generic choice of $(r_i)$ and $p$.
Suppose that the image of $\LL$ lies in $\theta_+$.
Proposition \ref{pic2move} then implies that 
for any nonempty $J\subset I_+$ the line bundle
$\OO(E-\sum_{i\in J}E_i)$
lies in $D$. 
Consider the Koszul complex on $\CC^n$ given $z_i,i\in I_+$.
It resolves the structure sheaf of a coordinate subspace
which is outside of $U$. This yields 
a long exact sequence of sheaves on $\PP_{\bf\Sigma}$
(see also \cite{BH}), which after twisting by $\LL$ becomes
$$
0\to \OO(E-\sum_{i\in I_+}E_i) \to \ldots \to \oplus_{i\in I_+} \OO(E-E_i) \to \LL\to 0.
$$
All but the last 
terms of this sequence lie in $D$, hence $\LL$ lies in $D$.

The calculation for the case when the image of $\LL$ is
in $p+\theta_-$ is completely analogous, as are the calculations
for when the point $p$ is moving in the opposite direction.
As such, they are left to the reader.

So now we have shown that all $\LL$ with the property
that their image $q$ in ${\rm Pic}_\RR(\PP_{\bf\Sigma})$
satisfies $\vert f(q-p) \vert \leq \frac 12$  lie in $D$. 
The second step of the construction 
is to move this whole slab in both directions
by making similar use of the Koszul relation for $I_+$
(or $I_-$, at this stage either one of the two suffices),
see Figure \ref{FigPic2-3}. We make use of the inequalities 
 $-1< f(-\sum_{i\in J}E_i) <  0$ for any 
$\emptyset\neq J\subseteq I_+$.
This finishes the proof.
\end{proof}

\begin{remark}\label{numpic2}
Similar to Remark \ref{numpic1},
it can be shown that, with the area form
that makes the volume of ${\rm Pic}_\RR(\PP_{\bf\Sigma})/{\rm Pic}(\PP_{\bf\Sigma})$ equal the 
order of the torsion subgroup of ${\rm Pic}(\PP_{\bf\Sigma})$,
the area of the parallelogram 
$P$ is $(n-2)!\Vol(\Delta)$. It can also be shown that the 
number of elements of $S$ equals 
$(n-2)!\Vol(\Delta)$.
This is, again, expected, since the number of elements
of $S$ needs to coincide with the rank of the Grothendieck group of $\PP_{\bf\Sigma}$.
\end{remark}

\begin{remark}
The case of toric varieties of Picard number at most two
has been settled in \cite{CM} by a different method.
Notice that \cite{CM} does not assume that the variety
is Fano.
We thank Rosa Mir\'o-Roig for bringing this article 
to our attention.
\end{remark}

\section{The case of del Pezzo toric stacks, the preliminaries}
\label{sec.dP.prelim}
In this section we will describe a construction in convex geometry which will eventually 
allow us to prove Conjecture \ref{fsec.conj} in the case of 
del Pezzo toric Deligne-Mumford stacks. The reader may
refer to Section \ref{sec.ex} for the example of this construction
for the case $n=5$.

Let $\Delta=A_1A_2\ldots A_n$ be a convex $n$-gon in $N=\ZZ^2$,
with the vertices counted clockwise, which contains $0$ in its
interior. Let ${\bf \Sigma}$ be the corresponding stacky fan and $\PP_{\bf \Sigma}$ the corresponding del Pezzo DM stack. As before,
we denote by $v_i$ the vector from $0$ to $A_i$ and by 
$E_i$ the corresponding elements of the Picard group of $\PP_{\bf\Sigma}$.

We will first introduce additional notation. Recall that by 
Proposition \ref{allaregiven} the Picard group
${\rm Pic}(\PP_{\bf\Sigma})$ is the quotient of $\ZZ^n$ with 
basis $E_i$ by the linear relations 
$$
\sum_{i=1}^n (w\cdot v_i) E_i 
$$
for $w\in N^*$. 

\begin{definition} 
We mod out by the span of the canonical divisor to
define the group $\widehat {\rm Pic}(\PP_{\bf\Sigma})$ 
by
$$\widehat {\rm Pic}(\PP_{\bf\Sigma})
=
{\rm Pic}(\PP_{\bf\Sigma})/\ZZ(\sum_{i=1}^n E_i).
$$
We denote by $\widehat {\rm Pic}_\RR(\PP_{\bf\Sigma})$
its tensor product with $\RR$. We denote by $\widehat E_i$
the images of $E_i$ in $\widehat {\rm Pic}(\PP_{\bf\Sigma})$
and $\widehat {\rm Pic}_\RR(\PP_{\bf\Sigma})$.
\end{definition}

\begin{definition}\label{defQ}
We denote by $Q$ the convex polytope in 
$\widehat {\rm Pic}_\RR(\PP_{\bf\Sigma})$
which is the convex hull of the points 
$\widehat E_I=\sum_{i\in I} \widehat E_i$ for all subsets 
$I \subseteq \{1,\ldots,n \}$.
\end{definition}

\begin{remark}
Polytope $Q$ is the Minkowski sum of line segments
$[0,\widehat E_i]$ and is thus a zonotope. The condition 
$\sum_{i=1}^n \widehat E_i=0$ ensures that the center of 
symmetry of $Q$ is $0$. 
\end{remark}

The following proposition describes the vertices of $Q$.
\begin{proposition}\label{QandForb}
The point $\widehat E_I$ is a vertex of $Q$ if and only if 
$I$ is a nonempty proper subset and $q_I=-\sum_{i\not\in I}E_i$
is a forbidden point. Equivalently, $\widehat E_I$ is a 
vertex of $Q$ if and only if 
the simplicial complex $C_I$ contains more than one 
connected component. 
\end{proposition}

\begin{proof}
It is clear that the set of vertices of $Q$ is a subset 
of the set of $\widehat E_I$. For $\widehat E_I$ to
be a vertex of $Q$ there has to exist a linear function
on $\widehat {\rm Pic}_\RR(\PP_{\bf\Sigma})$ which is 
maximized on it among other vertices of $Q$. In other
words, this linear function should take positive values 
on $\widehat E_i$ for $i\in I$ and negative values 
on $\widehat E_i$ for $i\not\in I$. Linear functions
$f$ on $\widehat {\rm Pic}_\RR(\PP_{\bf\Sigma})$
are collections of $n$ real numbers $r_i=f(\widehat E_i)$
that satisfy $\sum_{i=1}^n r_i=0$ and 
$\sum_{i=1}^n (w\cdot v_i) r_i = 0$
for all $w\in N^*$. In other words, $\sum_i r_i=0$ 
and 
\begin{equation}\label{rs}
\sum_{i=1}^n r_i v_i=0.
\end{equation}
Since $\sum_i r_i=0$,
it means that $I$ and its complement are nonempty.
We can then write \eqref{rs} as
\begin{equation}\label{bary}
\frac 1{\sum_{i\in I} r_i}\sum_{i\in I} r_i v_i 
= 
\frac 1{\sum_{i\not\in I} (-r_i)} \sum_{i\not\in I}(-r_i)v_i.
\end{equation}
So the existence
of the linear function with the required property is
equivalent to the condition that $I$ is proper and nonempty
and the relative interiors
of the convex hulls 
${\rm conv}(\{v_i,i\in I\})$ and ${\rm conv}(\{v_i,i\not\in I\})$
intersect. It is straightforward to see that in a convex 
polygon $\Delta$ the latter condition is equivalent
to $C_I$ having 
more than one connected component.
\end{proof}

\begin{remark}
Already in dimension three, the condition that 
relative interiors of
${\rm conv}(\{v_i,i\in I\})$ and ${\rm conv}(\{v_i,i\not\in I\})$
intersect is only necessary, but not sufficient to assure
that $C_I$ is not acyclic. Consequently, if dimension of 
$\Delta$ is bigger than two, then some vertices of $Q$ 
may not be images of forbidden points. 
\end{remark}

\begin{proposition}
For a vertex $\widehat E_I$ of $Q$ the image of the corresponding
forbidden cone $F_I$ under the projection 
${\rm Pic}_\RR(\PP_{\bf\Sigma})\to\widehat 
{\rm Pic}_\RR(\PP_{\bf\Sigma})$
is the opposite of the angle cone. In other words,
the image of $F_I$ is 
$$
\widehat E_I-\RR_{\geq 0}(Q-\widehat E_I).
$$
\end{proposition}

\begin{proof}
This follows immediately from the definition of $Q$ and the 
description of $F_I$ in Definition \ref{forbidden}.
\end{proof}

The argument of Proposition \ref{QandForb}
can be generalized to describe 
all faces of $Q$, and in
particular its facets.
\begin{proposition}\label{faceQ}
Faces of $Q$ correspond to ordered pairs of disjoint subsets 
$I$ and $J$ of $\{1,\ldots,n\}$, such that the relative interiors of 
the convex 
hulls ${\rm conv}(\{v_i,i\in I\})$ and ${\rm conv}(\{v_i,i\in J\})$
intersect. In particular, facets of $Q$ are in one-to-one correspondence with ordered pairs of intersecting diagonals in $\Delta$. Specifically, for a pair $(I,J)$, the corresponding face
is given by 
$$\theta_{I,J}
=\widehat E_I + \sum_{i\not\in I\cup J} [0,\widehat E_i].
$$
\end{proposition}

\begin{proof}
Faces of $Q$ are maximum sets on $Q$ of the linear functions on
$\widehat {\rm Pic}_\RR(\PP_{\bf\Sigma})$.
If a linear function $f$ is given by $r_i=f(\widehat E_i)$,
then consider the set $I$ of indices $i$ for which $r_i>0$
and the set $J$ of indices $j$ for which $r_j<0$. 
As before, we see that the relative interiors of the
convex hulls ${\rm conv}(\{v_i,i\in I\})$ and ${\rm conv}(\{v_i,i\in J\})$
intersect. Vice versa, any such intersection allows us
to define $r_i$ that give a linear function on 
$\widehat {\rm Pic}_\RR(\PP_{\bf\Sigma})$.
The maximum set of a linear function on the Minkowski sum 
of polytopes is the Minkowski sum of its maximum sets on
the individual polytopes. Hence, we get the formula for 
$\theta_{I,J}$. Finally, the pairs $(I,J)$ with the above property
are partially ordered by inclusion. This partial order
is the reverse of the inclusion order of the faces. So the facets 
correspond to the minimum pairs $(I,J)$ with 
${\rm conv}(\{v_i,i\in I\})\cap {\rm conv}(\{v_i,i\in J\})\neq \emptyset$,
and these are precisely the pairs of intersecting diagonals.
\end{proof}

\begin{remark}\label{noshape}
The condition 
${\rm conv}(\{v_i,i\in I\})\cap {\rm conv}(\{v_i,i\in J\})\neq \emptyset$,
on $I$ and $J$ above is purely combinatorial,
in the sense that it does not depend on the geometry of $\Delta$,
provided that $\Delta$ is convex. Specifically, it is equivalent
to the existence of indices $i_1,i_2\in I$ and $j_1,j_2\in J$,
such that $i_1,j_1,i_2,j_2$ are counted clockwise, if one 
sets $\{1,\ldots,n\}$ in a clockwise circle.
\end{remark}

Our next goal is to construct a polytope $\widehat P$ 
in $\widehat {\rm Pic}_\RR(\PP_{\bf\Sigma})$ with centrally 
symmetric faces which
has the peculiar property that the midpoints of all facets
of $\widehat P$ are vertices of $Q$ and all vertices of $Q$ are 
midpoints of some faces of $\widehat P$. This is the key ingredient of 
the argument of this paper. This polytope $\widehat P$
will also be a zonotope and it will have a combinatorial
structure that is identical to that of $Q$.

Consider the stacky fan ${\bf\Sigma}_1$ 
in $N$ given by the rays
$$(t_1,t_2,\ldots,t_n) = (v_1-v_n,v_2-v_1,\ldots, v_n-v_{n-1}).$$
It will be convenient for us to consider our subscripts to
be elements of $\ZZ/n\ZZ$, so that we can write the
above equation simply as $t_i=v_{i}-v_{i-1}$.
Note that the convexity of $\Delta$ assures that $t_i$ are 
counted clockwise, although we can no longer assume that
they are vertices of a convex polytope. Consider a 
collection of positive numbers $\phi_i$ such that 
$\frac 1{\phi_i} t_i$ are vertices of a convex polytope.
By scaling $\phi_i$ we may arrange so that 
\begin{equation}\label{sum1}
\sum_{i=1}^n \phi_i=1.
\end{equation}

\begin{remark}
There is a fairly natural choice of $\phi$ given by 
$\phi_i=v_i\wedge v_{i-1}$ for some area form
$N_\RR\wedge N_\RR\to \RR$. This corresponds to
considering the dual polytope $\Delta^*$ and placing
it in $N$ via the identification of $N$ and $N^*$ via
the above form. It can consequently be scaled to get
$\sum_{i=1}^n \phi_i=1$. On the other hand, the
arguments of the paper go through for any convex
$\phi$, even if some $\phi_i$ are negative.
\end{remark}

\begin{definition}\label{defhatP}
We define the zonotope $\widehat P$ in 
$\widehat {\rm Pic}_\RR(\PP_{\bf\Sigma})$
which is the Minkowski sum of segments
$[\widehat t_i,-\widehat t_i]$ where 
$\widehat t_i$ are given
by 
$$
\widehat t_{i+1}-\widehat t_{i} = \widehat E_i,~{\rm for~all~} i\in 
\ZZ/n\ZZ
$$
and 
$$
\sum_{i=1}^n \phi_i\widehat t_i=0.
$$
\end{definition}

\begin{remark}
It is easy to see that Definition \ref{defhatP} determines
$\widehat t_i$ and hence $\widehat P$ uniquely. 
Indeed, the first set of equations can be solved because
$\sum_{i=1}^n \widehat E_i=0$. It 
determines $\widehat t_i$
uniquely up to an addition of an element of $\widehat {\rm Pic}_\RR(\PP_{\bf\Sigma})$ and then the last relation
removes the remaining ambiguity uniquely in view of 
\eqref{sum1}. Specifically, we get 
$$
\widehat t_i = \frac 1n\Big( \sum_{j=0}^{n-1} j\widehat E_{i+j} 
-\sum_{k=1}^n\sum_{j=0}^{n-1} j \phi_k \widehat E_{k+j}\Big)
$$
but we will not need this form of the solution. 
\end{remark}

We can explicitly describe the face structure of $\widehat P$.
\begin{proposition}\label{facesofP}
Faces of $\widehat P$ correspond to ordered pairs of disjoint subsets 
$I$ and $J$ of $\ZZ/n\ZZ$, such that the relative interiors of 
the convex 
hulls ${\rm conv}(\{v_i,i\in I\})$ and ${\rm conv}(\{v_i,i\in J\})$
intersect. In particular, facets of $\widehat P$ are in one-to-one correspondence with ordered pairs of intersecting diagonals in $\Delta$. Specifically, for a pair $(I,J)$, the corresponding face
is given by 
$$\theta_{I,J}
=\sum_{i\in I} \widehat t_i
-\sum_{i\in J}\widehat t_i + \sum_{i\not\in I\cup J} [-\widehat t_i,
\widehat t_i].
$$
\end{proposition}

\begin{proof}
Consider a supporting function $f$ with $r_i=f(\widehat t_i)$. 
In view of the definition of $\widehat t_i$,
these $r_i$ satisfy a three-dimensional space of linear relations
$$
\sum_{i=1}^n \phi_i r_i = 0
$$
and
$$
\sum_{i=1}^n (w\cdot v_i) (r_{i+1}-r_i) = 0
$$
for all $w\in N^*$. The latter relation can be rewritten 
as 
$$
\sum_{i=1}^n r_i t_i = 0.
$$
This can be then thought of as a linear relation 
\begin{equation}\label{phph}
\sum_{i=1}^n (\phi_i r_i) (\frac 1{\phi_i}t_i) = 0
\end{equation}
on points $\frac 1{\phi_i}t_i$ in $N_\RR$ with 
$\sum_{i} \phi_i r_i = 0$. If $I$ is the set of $i$
with $r_i>0$  (and hence $\phi_ir_i>0$)
and $J$ is the set of $i$ with $r_i<0$,
then, similarly to \eqref{bary}, we can see \eqref{phph}
as a statement that the relative interiors of 
the convex hulls of $\frac 1{\phi_i}t_i,i\in I$ and 
$\frac 1{\phi_i}t_i,i\in J$ intersect. In view of Remark \ref{noshape}
we can replace $\frac 1{\phi_i}t_i$ by $v_i$.

The calculation of the maximum set of $f$ on $\widehat P$ 
is then straightforward.
The statement about facets is also clear.
\end{proof}

\begin{remark}
A reader familiar with Gale duality will notice that 
the proofs of Propositions \ref{faceQ} and \ref{facesofP}
can be stated naturally in its terms, since the facet structure
of a zonotope is encoded by the linear combinations 
in the Gale dual picture. However, we preferred to give 
a direct argument to avoid introducing additional terminology.
\end{remark}

The main properties of $\widehat P$ are summarized in the following
proposition.
\begin{proposition}\label{allabouthatP}
The polytope $\widehat P$ is centrally symmetric.
All vertices of $Q$ are midpoints of some faces of $\widehat P$. 
A vertex $\widehat E_I$ of $Q$ is a midpoint of a facet of 
$\widehat P$ if and only if $C_I$ has exactly two connected components. The midpoint of every facet of $\widehat P$ is 
a vertex of $Q$.
\end{proposition}

\begin{proof}
Denote by $[i,j)$ the set of indices in $\ZZ/n\ZZ$ starting
from $i$ (included) and ending with $j$ (excluded), counted 
clockwise. Then 
\begin{equation}\label{tele}
\widehat E_{[i,j)} = \widehat t_j - \widehat t_i.
\end{equation}

Every proper subset $I$ of $\ZZ/n\ZZ$ such that simplicial 
complex $C_I$ has nontrivial reduced homology can
be uniquely written as a disjoint union of $l$ intervals
$[i_k,j_k)$, $k=1,\ldots,l$ with $l\geq 2$. Equation 
\eqref{tele} then gives
$$
\widehat E_I =
\sum_{k=1}^l\widehat t_{j_k}-
 \sum_{k=1}^l \widehat t_{i_k},
$$
which is the midpoint of a face of $\widehat P$ by
Proposition \ref{facesofP}. In particular, this face is a facet
of $\widehat P$ if and only if $l=2$, which is equivalent 
to $C_I$ having two connected components. Finally, all 
facets of $\widehat P$ are obtained by this procedure.
\end{proof}

\begin{corollary}\label{Pacyclic}
The interior of $\widehat P$ lies outside of the images of 
the forbidden cones $F_I$ for all proper  
subsets $I$ of $\{1,\ldots,n\}$ with non-acyclic $C_I$.
\end{corollary}

\begin{proof}
For each forbidden cone $F_I$
consider the corresponding point $\widehat E_I$ on 
the boundary of $\widehat P$. 
By Proposition \ref{allabouthatP} the polytope $Q$ 
is contained in $\widehat P$, so any supporting hyperplane
of $\widehat E_I$ for $\widehat P$ is also a supporting
hyperplane for it for $Q$. It remains to observe that the
image of the forbidden cone for $Q$ lies on the side
of this hyperplane away from $Q$ and hence 
away from the interior 
of $\widehat P$ by Proposition \ref{QandForb}.
\end{proof}

\begin{proposition}\label{movinglemma}
Let $I=[i_1,j_1)\sqcup [i_2,j_2)$ with $i_1,j_1,i_2,j_2$ 
indexed clockwise and let 
$$\theta_I = \widehat t_{j_1} + \widehat t_{j_2} 
-\widehat t_{i_1} - \widehat t_{i_2}
+\sum_{k\neq i_1,i_2,j_1,j_2}[-\widehat t_k,\widehat t_k]
$$
be the facet of $\widehat P$ that contains $\widehat E_I$
as a midpoint. Then the shifts of the relative interiors
of $\theta_I$ by $-2\widehat E_{[i_1,j_1)}$, 
$-2\widehat E_{[i_2,j_2)}$, $2\widehat E_{[j_1,i_2)}$
and $2\widehat E_{[j_2,i_1)}$ are contained in the interior
of $\widehat P$. In addition, the shift of the relative
interior of $\theta_I$ by $-2\widehat E_I=2\widehat E_{\bar I}$
is the opposite face $-\theta_I=\theta_{\bar I}$ of $\widehat P$. 
\end{proposition}

\begin{proof}
We will prove the last of the four statements. The proof of 
the other three is completely analogous and is left to the reader.
In view of the equation \eqref{tele}, the shift of the relative interior of 
$\theta_I$ by $-2\widehat E_{[j_2,i_1)}$ is given 
by 
$$
\widehat t_{i_1} -  \widehat  t_{j_2} + \widehat t_{j_1} 
- \widehat t_{i_2}+
\sum_{k\neq i_1,i_2,j_1,j_2}(-\widehat t_k,\widehat t_k).
$$
Every point $p$ of this set clearly lies in $\widehat P$, so it remains to show that it does not lie on the  boundary of $\widehat P$. Assume
the contrary. For any  index $k\neq i_1,i_2,j_1,j_2$ the point 
$p$ can be moved by $\epsilon \widehat t_k$ for small $\vert \epsilon\vert$, so that the result is still in $\widehat P$.
Consequently, any supporting hyperplane at $p$ 
should be given by a linear function $f$ with $f(\widehat t_k)=0$
for $k\neq i_1,i_2,j_1,j_2$. This leads  to a
statement that the interiors of the segments 
$[\frac 1{\phi_{i_1}}t_{i_1},\frac 1{\phi_{j_1}}t_{j_1}]$
and $[\frac 1{\phi_{i_2}}t_{i_2},\frac 1{\phi_{j_2}}t_{j_2}]$  intersect,
which is false. 

The statement about the shift by $-2\widehat E_I$ is an easy 
calculation which we leave to the reader.
\end{proof}

\section{The case of del Pezzo toric stacks, the full strong exceptional collection}\label{sec.dP}
In this section we will prove Conjecture \ref{fsec.conj} for
toric del Pezzo Deligne-Mumford stacks.

We will be using the notations of the preceding section.
First, we will define a polytope $P$ in 
${\rm Pic}_\RR(\PP_{\bf\Sigma})$ as follows.
Fix a collection of positive numbers $r_i$ such that 
$\sum_{i=1}^n r_i=1$ and $\sum_{i=1}^n r_i v_i=0$.
This collection defines a linear function $f$ on 
${\rm Pic}_\RR(\PP_{\bf\Sigma})$ 
with $f(E_i)=r_i$.
\begin{definition}\label{P.def}
We define a convex polytope $P$ in 
${\rm Pic}_\RR(\PP_{\bf\Sigma})$
which consists of points of $x$ 
with 
$\vert f(x)\vert \leq \frac 12$
such that the image of $x$ in $\widehat
{\rm Pic}_\RR(\PP_{\bf\Sigma})$
lies in $\frac 12 \widehat P$.
\end{definition}

We pick a generic  $p\in {\rm Pic}_\RR(\PP_{\bf\Sigma})$.
As in Theorem \ref{Pic2} we consider the 
set $S$ of line bundles $\LL$ such that their
image in ${\rm Pic}_\RR(\PP_{\bf\Sigma})$ lies in
$P+p$.
\begin{proposition}\label{strongexc}
The set $S$ forms a strong exceptional collection.
\end{proposition}

\begin{proof}
We simply need to check that the differences of any two
line bundles in $\LL$ lie outside of all forbidden cones.
Since $p$ is generic, $p+P$ has no lattice points on 
the boundary, consequently, the differences of line bundles
in $S$ map to the interior of $p+P+(-p-P)$. Because 
$P$ is centrally symmetric, these differences are then 
in the interior of $2P$. Points $x=\sum_{i=1}^n x_i E_i$
in the interior of $2P$ 
satisfy  $\sum_{i=1}r_ix_i > -1=f(-\sum_{i=1}^nE_i)$, which shows that they
lie outside the forbidden cone  for $I=\emptyset$. 
To show that the other forbidden cones $F_I$ do not intersect the 
interior of $2P$ consider their projections to 
$\widehat {\rm Pic}_\RR(\PP_{\bf\Sigma})$.
By Corollary \ref{Pacyclic}, 
they do not intersect the interior of $\widehat P$, 
which is precisely the projection of the interior of $2P$.
\end{proof}

We are now ready to prove the main result of this paper,
which is to show that $S$ is a full strong exceptional collection.
\begin{theorem} \label{delPezzo}
For a generic $p\in {\rm Pic}_\RR(\PP_{\bf\Sigma})$
the set $S$ of line bundles $\LL$ that map inside $P+p$ forms a 
full strong exceptional collection on $\PP_{\bf\Sigma}$.
\end{theorem}

\begin{proof}
Denote by $D$ the subcategory of the derived category of the 
coherent sheaves on $\PP_{\bf\Sigma}$ generated by 
the elements of $S$. 

We will first show that $D$ contains all line bundles $\LL$ 
whose image $x$ in $ {\rm Pic}_\RR(\PP_{\bf\Sigma})$
satisfies $\vert f(x-p)\vert \leq \frac 12$. Fix one such 
$\LL$ and the corresponding $x$.
Since $p$ is chosen generic, we may safely assume that 
no lattice points $y$ satisfy $\vert f(y-p)\vert = \frac 12$.
Consider the set of elements $p_1$ of  
${\rm Pic}_\RR(\PP_{\bf\Sigma})$ 
such that $f(p_1)=0$ and $p+p_1+P$ contains $x$ in 
its interior. This is a relatively open subset in the codimension
one subspace of ${\rm Pic}_\RR(\PP_{\bf\Sigma})$
characterized by $f(p_1)=0$. Pick a generic such $p_1$
and consider for all $t$ from $0$ to $1$ the collection
$S(t)$ of line bundles $\LL$ whose images in 
${\rm Pic}_\RR(\PP_{\bf\Sigma})$ lie in $p+tp_1+P$.

The assumption that $p$ and $p_1$ are generic implies 
that for all $t$ there are no lattice points on the codimension
two or higher faces of $p+tp_1+P$. Indeed, since there
are no lattice points that satisfy $\vert f(y-p) \vert = \frac 12$
we may assume that the face in question is a shift 
of the preimage of a face in $\frac 12\widehat P$ of codimension two 
or more. Since for a given $x$ 
the union of all possible $p+tp_1+P$ (as $p_1$ and $t$ vary) 
is a bounded
subset of ${\rm Pic}_\RR(\PP_{\bf\Sigma})$, there are only finitely
many lattice points that could be on the sides of 
some $p+tp_1+P$. For each such lattice point $y$
and each face $\theta$ of $P$ the condition $y\in p+tp_1+\theta$
cuts out a space of codimension at least two 
in the space of all possible $tp_1$. This in turn gives a codimension
at least one space of $p_1$ such that for some $t\in [0,1]$
there holds $y\in p+tp_1+\theta$. 

Similarly we may also assume that the boundaries
of $p+P$ and $p+p_1+P$
contain no lattice points. Assume that the line bundle
$\LL$ does not 
lie in $D$.  Since there are only finitely many lattice
points in $p+[0,p_1] + P$, the segment $[0,1]$ is subdivided
into a finite number of segments on which the set of lattice
point in $p+tp_1+P$ is constant. 
Consequently, we may consider the smallest value $t_1$ of $t$
such that there is a lattice point $x_1$ in $p+tp_1+P$ such 
that there is a line bundle $\LL_1$ which maps to it and 
does not lie in $D$. 

By the above argument, $x_1$ lies in a face $\theta$ 
of codimension one
of $p+t_1p_1+P$ which is the preimage of the shift of a 
facet $\widehat p + t_1\widehat p_1+\frac 12\widehat \theta$ of 
$\widehat p +t_1\widehat p_1+\frac 12 \widehat P$ under 
the projection map. The corresponding set $I$ is the union of
two intervals 
$$
I=[i_1,j_1)\sqcup [i_2,j_2).
$$
The facet $\theta$ contains $x_1$ and the points 
$p+t_1p_1 - \frac 12 \sum_{i\not\in I} E_i$,
$p+t_1p_1 + \frac 12 \sum_{i\in I} E_i$. 
We have
$$
f(p+t_1p_1 + \frac 12 \sum_{i\in I} E_i) - 
f(p+t_1p_1 - \frac 12 \sum_{i\not\in I} E_i)  = \frac 12\sum_{i=1}^n 
r_i=\frac 12.
$$
This allows us to decompose $p+t_1p_1+\theta$ into 
the union of two centrally symmetric polytopes, with 
centers of symmetry 
$p+t_1p_1 + \frac 12 \sum_{i\in I} E_i $
and 
$p+t_1p_1 - \frac 12 \sum_{i\not\in I} E_i$,
as in  Figure \ref{FigPic2-2}. To determine which
polytope $x_1$ belongs to, we need to compare
$f(x_1-p+\sum_{i\not\in I}E_i)$ to $\frac 12$.
We can safely assume that it is not equal to $\frac 12$
since $p$ is picked to be generic.

{\bf Case $f(x_1+\sum_{i\not\in I}E_i-p)<\frac 12$.}
In this case, the points 
$$x_1+\sum_{i\in [j_1,i_2)} E_i,~~~
x_1+\sum_{i\in [j_2,i_1)} E_i 
$$
lie in the interior of $p+t_1p_1+P$ and the point
$$x_1+\sum_{i\not\in I} E_i$$ 
lies in the interior of the opposite facet 
$2p+2t_1p_1-\theta$ of $p+t_1p_1+P$.
Indeed, the projections of the first two points to $\widehat 
{\rm Pic}_\RR(\PP_{\bf\Sigma})$ lie in the interior of
and the projection of the third point lies on the opposite facet
$\widehat p+t_1\widehat p_1-\theta_I$ of 
$\widehat p+t_1\widehat p_1 + \frac 12\widehat P$
by Proposition \ref{movinglemma}. It remains to check
the property $\vert f(y-p)\vert <\frac  12$ for each of these 
points. Since $f(x_1-p)>-\frac 12$ and $f(E_i)=r_i>0$,
we only need to check $f(y-p)<\frac 12$. The largest of these
values occurs for $f(x_1+\sum_{i\not\in I}E_i)$, which is less than 
$\frac 12$, by our assumption.

Observe that for small $\epsilon>0$ the three points of interest
lie in the interior of $p+(t_1-\epsilon)p_1+P$. Indeed,
by our assumption of minimality of $t_1$ the point $x_1$
does not lie in $p+(t_1-\epsilon)p_1+P$ for small positive 
$\epsilon$, which means that the value of the supporting 
function of $\theta$ is negative on $p_1$. As a result,
every point in the interior of the opposite face will lie in the 
interior of $p+(t_1-\epsilon)p_1+P$ for small $\epsilon>0$.
By the minimality assumption we conclude that the line
bundles $\LL_1(\sum_{i\in[j_1,i_2)} E_i)$,
$\LL_1(\sum_{i\in [j_2,i_1)}E_i)$ and $\LL_1(\sum_{i\not\in I} E_i)$
lie in the category $D$. We can then write a Koszul 
sequence on $\PP_{\bf\Sigma}$
$$
0\to \LL_1\to 
\LL_1(\sum_{i\in[j_1,i_2)} E_i)\oplus\LL_1(\sum_{i\in [j_2,i_1)}E_i)
\to \LL_1(\sum_{i\not\in I} E_i)\to 0
$$
which is exact, since the divisors $\sum_{i\in[j_1,i_2)} E_i$
and $\sum_{i\in[j_2,i_1)} E_i$ have no common zeroes 
in $\PP_{\bf\Sigma}$. This shows that $\LL_1$ is in $D$,
contradiction.

{\bf Case $f(x_1+\sum_{i\not\in I}E_i-p)>\frac 12$.}
Observe that 
$$
f(x_1-\sum_{i\in I}E_i -p) = f(x_1+\sum_{i\not\in I}E_i-p) -
\sum_{i=1}^n f(E_i) > \frac 12 - 1 = -\frac 12.
$$
Similarly to the previous case we can show, using Proposition
\ref{movinglemma}, that $\LL_1(-\sum_{i\in [i_1,j_1)}E_i)$,
$\LL_1(-\sum_{i\in [i_2,j_2)}E_i)$ and 
$\LL_1(-\sum_{i\in I}E_i)$ are in $D$. The Koszul short
exact sequence
$$
0\to \LL_1(-\sum_{i\in I} E_i)\to 
\LL_1(-\sum_{i\in[i_1,j_1)} E_i)\oplus\LL_1(-\sum_{i\in [i_2,j_1)}E_i)
\to \LL_1\to 0
$$
then finishes the argument. 

So we have shown that all $\LL$ which map to $x$ with
$\vert f(x-p)\vert \leq \frac 12$ lie in $D$.  By looking at shifts of 
any short exact Koszul complex we can then extend the range of values of $f(x)$ in both directions to finish the argument.
\end{proof}

\begin{remark}
In our desire to focus on the basic features of the problem,
we have restricted our attention to the del Pezzo case,
as opposed to the nef del Pezzo case. It is likely that
a slight modification of our argument would allow one to
handle the nef del Pezzo case as well. Indeed, our construction 
of polytope $\widehat P$ is continuous in vertices of $\Delta$,
and we should be able to take a limit as a side of $\Delta$ flattens.
We might no longer be able to guarantee in Proposition \ref{movinglemma} that the shifts of the interior of the face lie 
in the interior of $\widehat P$, but it is still likely that in 
the proof of Theorem \ref{delPezzo} the new points 
in $p+tp_1+ P$ could be expressed in terms of the  old ones
as $t$ increases. 
Another reason not to consider the nef case in this article
is that the work of Kawamata \cite{Ka2} assures that the derived
category of a nef del Pezzo toric stack is equivalent to that 
of the corresponding $K$-equivalent del Pezzo toric stack
obtained by only keeping the vertices of $\Delta$. Hence our 
results already guarantee the existence of a strong 
exceptional collection of objects in the derived category
of a nef del Pezzo toric stack, although these objects 
might not be line bundles.
\end{remark}

\section{The case of ${\rm rk(Pic)}=3$ and $\dim =2$}\label{sec.ex}
In this section we will illustrate the result and construction
of Sections \ref{sec.dP.prelim} and  \ref{sec.dP} 
in the case of $n=5$.

Let $\Delta=A_1A_2A_3A_4A_5$ be a convex pentagon
in $N=\ZZ^2$, with the vertices counted clockwise,
which contains $0$ in its interior. Let ${\bf\Sigma}$ be 
the corresponding stacky fan and $\PP_{\bf\Sigma}$ the corresponding del Pezzo DM stack.
As before, we denote by $v_i$ the vector from $0$ to $A_i$
and by $E_i$ the corresponding elements of the Picard 
group. We will introduce the notation
$$\widehat {\rm Pic}_\RR(\PP_{\bf\Sigma})
={\rm Pic}_\RR(\PP_{\bf\Sigma})/\RR K$$
where $K$ is the canonical class. 
We will abuse the notation and denote by $E_i$
the image of $\OO(E_i)$ in 
${\rm Pic}_\RR(\PP_{\bf\Sigma})$.
We will use the notation $\widehat E_i$ for the image of $E_i$
in $\widehat {\rm Pic}_\RR(\PP_{\bf\Sigma})$.

The polytope $Q$ 
in $\widehat{\rm Pic}_\RR(\PP_{\bf\Sigma})$ is given in Figure
\ref{Fig1}. 
It is a convex centrally symmetric 
$10$-gon.

\begin{figure}[tbh] 
\begin{center}
\includegraphics[scale = .7]{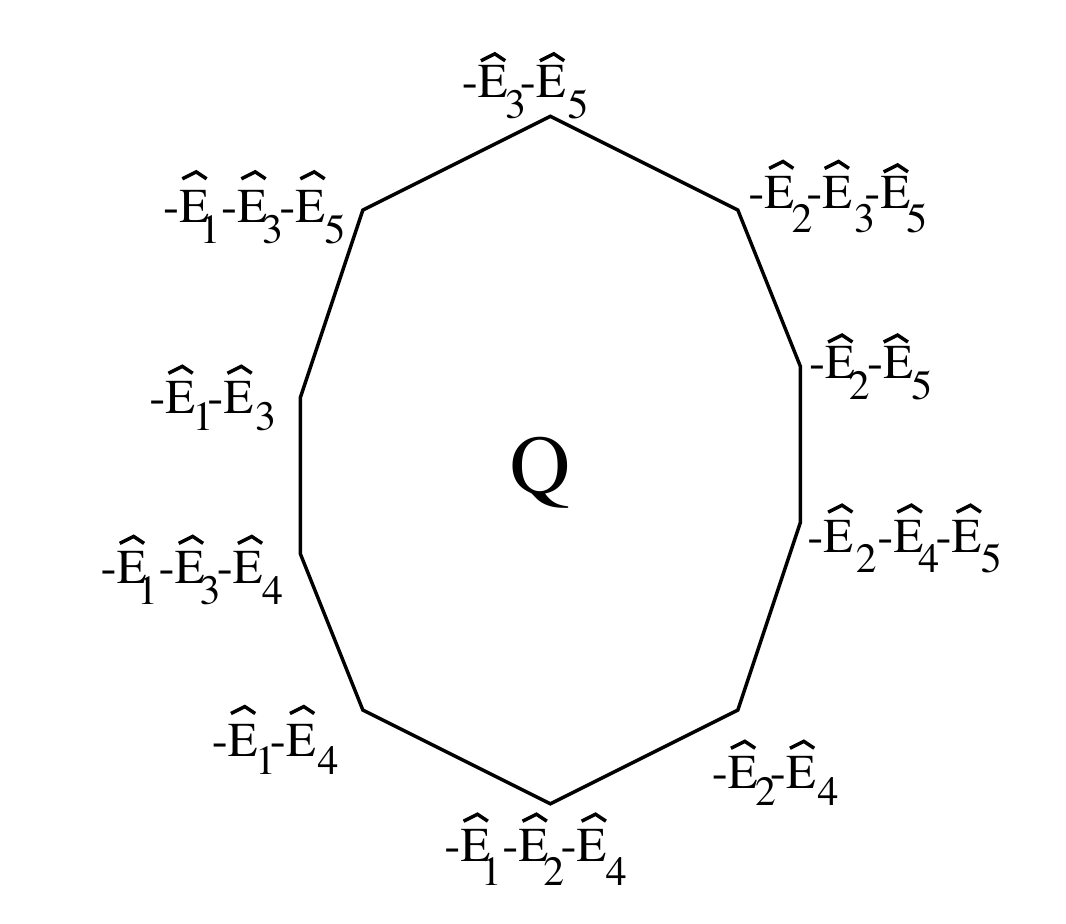}
\end{center}
\caption{\label{Fig1}}
\end{figure}

The projections of the forbidden cones for proper subsets 
$I\subset \{1,\ldots, 5\}$ are given in Figure \ref{Fig2}.
The complement of it is the acyclic region, in the sense that 
any line bundle $\LL$ which projects to it has no middle
cohomology.

\begin{figure}[tbh] 
\begin{center}
\includegraphics[scale = .7]{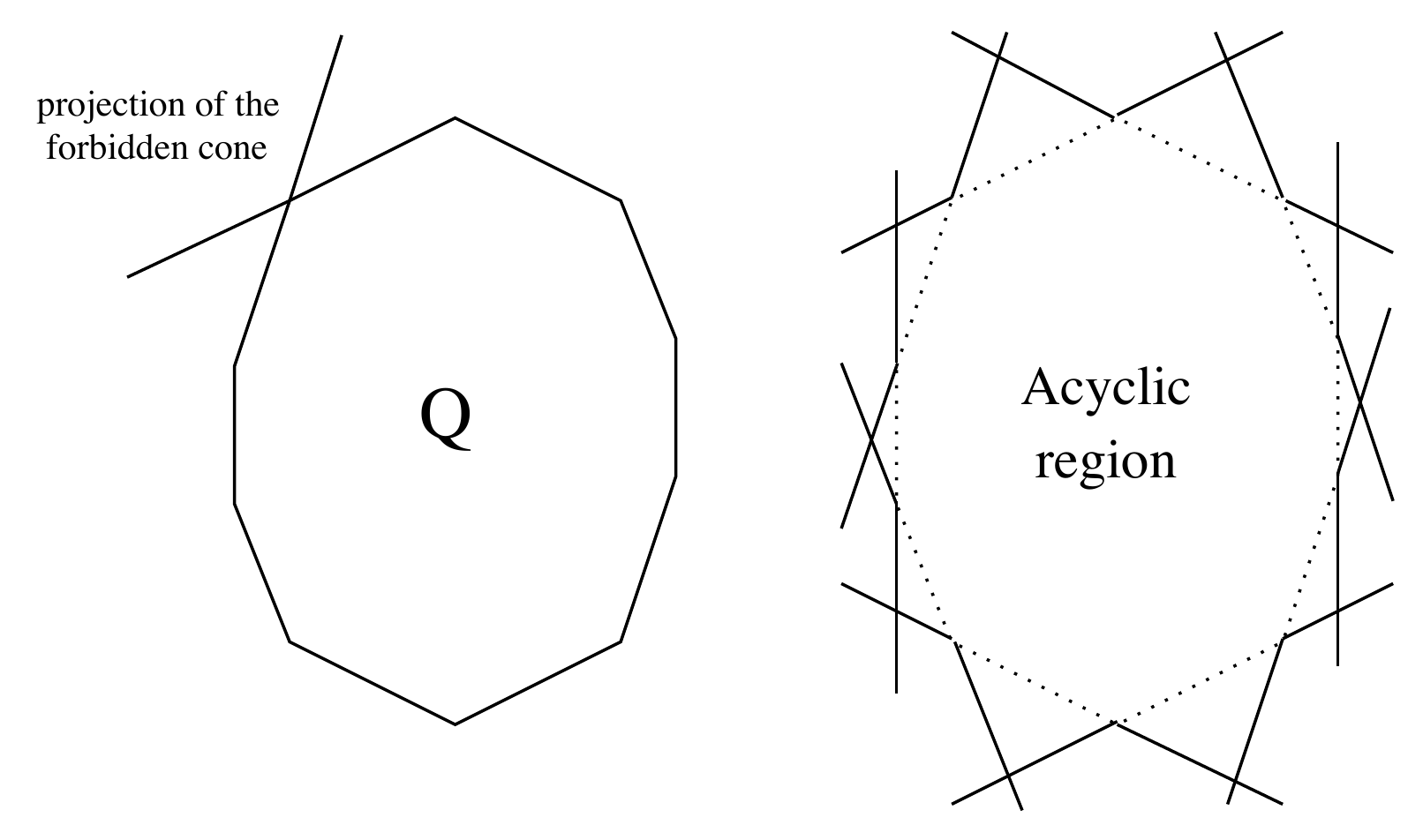}
\end{center}
\caption{\label{Fig2}}
\end{figure}

For any point $V$ in $\widehat{\rm Pic}_\RR(\PP_{\bf\Sigma})$, we can consider the points obtained from it by flipping it across the vertices of $Q$. A flip of a point $A$ across a point $B$ is $2B-A$.
It is easy to see that $\sum_{i=1}^5 \widehat E_i=0$ implies that 
if one starts with a point $V$ and flips it across $-\widehat E_1-
\widehat E_3$, then the $10$-th vertex is again $V$,
and the ten vertices are 
$$
\begin{array}{c}
V, -V-2\widehat E_3-2\widehat E_5,V-2\widehat E_2,-V-2\widehat E_5, V-2\widehat E_2-2\widehat E_4,
\\
-V, V+2\widehat E_3+2\widehat E_5, -V+2\widehat E_2, V+2\widehat E_5, -V+2\widehat E_2+2\widehat E_4.
\end{array}
$$

It is a priori not obvious that one can pick $V$ in such 
a way that the resulting ten points form vertices of a 
convex polytope that contains $Q$. However, by Proposition
\ref{allabouthatP} there exists a convex polygon $\widehat P$ such 
that the midpoints of its edges are the vertices of $Q$. Hence
one can pick $V$ so that 
the above 10 points are vertices of $\widehat P$.
In particular, the interior of $\widehat P$ lies in the acyclic region,
see Figure \ref{Fig3}.

\begin{figure}[tbh] 
\begin{center}
\includegraphics[scale = .7]{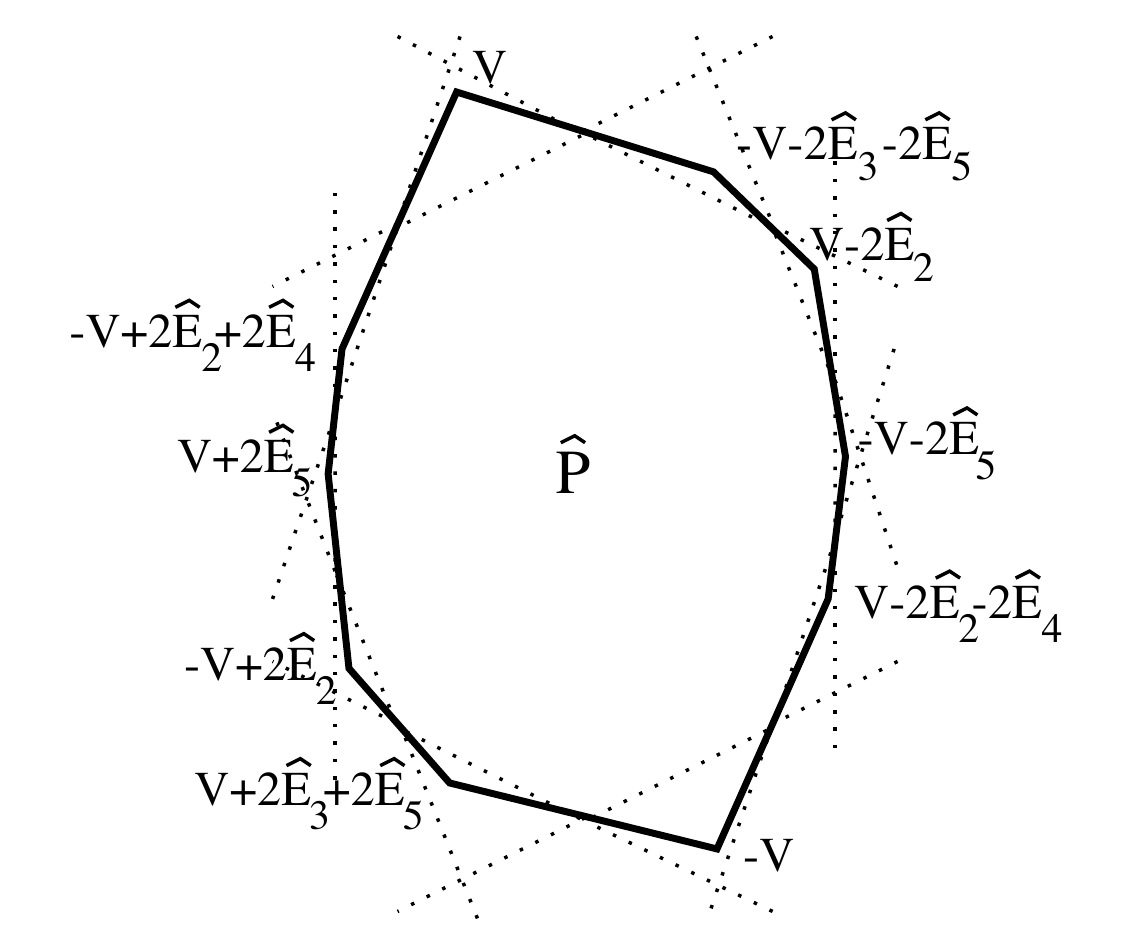}
\end{center}
\caption{\label{Fig3}}
\end{figure}

Proposition \ref{movinglemma} can be stated as 
follows. It will be convenient to consider the
indices $i$ of $A_i$ to lie in $\ZZ/5\ZZ$.
\begin{proposition}\label{move}
For an edge of $\widehat P$ that contains 
$-\widehat E_{i-1}-\widehat E_{i+1}$, the translates of its interior by  
$2\widehat E_{i-1}$, $2\widehat E_{i+1}$, $-2\widehat E_{i}$ and $-2(\widehat E_{i-2}+\widehat E_{i+2})$
lie in the interior of $\widehat P$.
For an edge of $\widehat P$ that contains $-\widehat E_{i}-\widehat E_{i-2}-\widehat E_{i+2}$,
the translates of its interior by $2\widehat E_i$, $2(\widehat E_{i-2}+\widehat E_{i+2})$,
$-2\widehat E_{i-1}$ and $-2\widehat E_{i+1}$ lie in the interior of $\widehat P$.
\end{proposition}

For what follows we pick and fix a generic collection of positive numbers $r_i, i=1,\ldots, 5$, 
such that $\sum_i r_i=1$ and $\sum_i r_iv_i=0$. This collection
gives a linear function $f$ on
${\rm Pic}_\RR(\PP_{\bf\Sigma})$ by $f(E_i)=r_i$.

The convex polytope $P$ in ${\rm Pic}_\RR(\PP_{\bf\Sigma})$
given by the inequalities $\vert f(x)\vert\leq \frac 12$
and the condition that the image of $x$ in $\widehat{\rm Pic}_\RR(\PP_{\bf\Sigma})$
lies in $\frac 12\widehat P$ is depicted in Figure \ref{Fig4}.
The polytope $P$ has two $10$-gonal faces
and $10$ parallelogram faces that are preimages 
of the sides of the pentagon $\frac 12 \widehat P$, see
Figure \ref{Fig4}.

\begin{figure}[tbh] 
\begin{center}
\includegraphics[scale = .7]{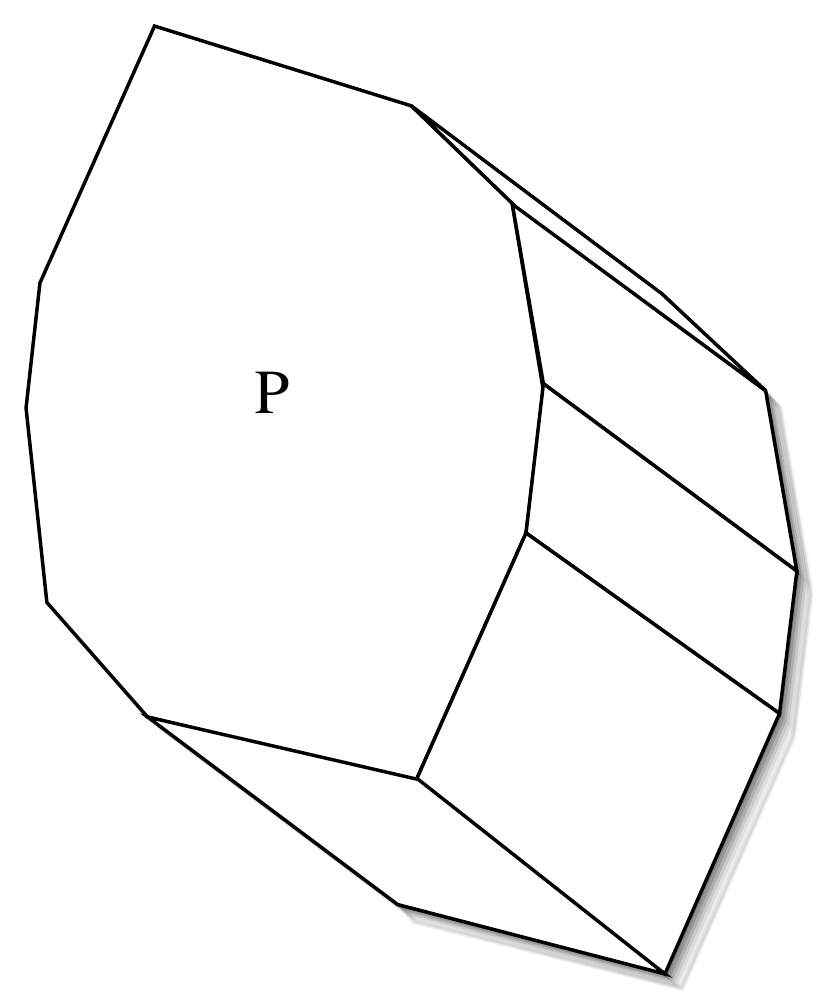}
\end{center}
\caption{\label{Fig4}}
\end{figure}

The proof of Theorem \ref{delPezzo} can be visualized as
follows. We place the polytope $P$ somewhere generically
in ${\rm Pic}_\RR(\PP_{\bf\Sigma})$ by considering 
its shift $p+P$. The set of line bundles
whose images in ${\rm Pic}_\RR(\PP_{\bf\Sigma})$ lie 
in $p+P$ form a strong exceptional collection. Indeed, the differences
avoid the forbidden cone $F_\emptyset$ with the vertex 
$q_\emptyset=-\sum_{i=1}^5 E_i$
because they have $f(\cdot)>-1$, and they avoid the other
forbidden cones because their image in 
$\widehat {\rm Pic}_\RR(\PP_{\bf\Sigma})$ is contained 
in $\widehat P$ and hence in the acyclic region. 
We then define the category $D$ generated by the line bundles 
in this strong exceptional collection. We first move the polytope
$p+P$ in generic directions that are parallel to its 10-gonal facets.
As the polytope moves, any new points can be connected
to already covered points by means of Proposition \ref{move}.
This in turn leads to Koszul complexes which allow one
to show that the corresponding line bundles lie in $D$. 
After we have guaranteed that all points between the 
supporting planes of the 10-gonal facets of $p+P$ correspond
to line bundles in $D$, we use Koszul complexes to 
extend in the orthogonal direction, as in the 
second panel of Figure \ref{FigPic2-3}.

\section{Comments}\label{comments}
It is natural to try to apply 
the techniques of this paper to the general
case of King's conjecture. For an arbitrary rank of the Picard 
group, and arbitrary dimension, one can still define the polytope 
$Q$ in $\widehat{\rm Pic}_\RR(\PP_{\bf\Sigma})$ as 
the Minkowski sum of $[0,\widehat E_i]$.  
One then wants to construct a polytope $\widehat P\supseteq Q$ 
with the property that all vertices of $Q$ that correspond to 
forbidden cones are midpoints of some of the faces of $\widehat P$
and that midpoints of all facets of $\widehat P$ are images of 
the vertices of the forbidden cones.

It is not a priori clear that $\widehat P$ should be a zonotope, 
although this is a plausible assumption. However even assuming 
that it is, the combinatorics of it is generally unclear and remains
the key challenge. Once the polytope $\widehat P$ is constructed, we can define the polytope $P$ in
${\rm Pic}_\RR(\PP_{\bf\Sigma})$ as in Section \ref{sec.dP}.
It remains to be seen whether this approach will
lead to a proof of Conjecture \ref{fsec.conj},
but it appears promising.

Even in its current state the paper can be applied to the study
of (noncompact) toric Calabi-Yau threefolds, which are defined
by triangulations of the polygon $(\Delta,1)$ in $N\oplus \ZZ$.

\end{document}